\documentclass[aos,preprint]{imsart}
\setattribute{journal}{name}{}
\RequirePackage[OT1]{fontenc}
\RequirePackage{amsthm,amsmath,amssymb}
\RequirePackage{enumerate,graphicx,float,subcaption,multirow}
\RequirePackage[table,xcdraw]{xcolor}
\RequirePackage{natbib}
\RequirePackage[colorlinks,citecolor=blue,urlcolor=blue]{hyperref}

\definecolor{blue}{rgb}{0,0,1}

\definecolor{red}{rgb}{1,0,0}

\startlocaldefs
\numberwithin{equation}{section}
\theoremstyle{plain}
\newtheorem{theorem}{Theorem}
\newtheorem{result}{Result}
\newtheorem{lemma}{Lemma}
\endlocaldefs

\begin{document}

\begin{frontmatter}
\title{Asymptotic properties of Principal Component Analysis and shrinkage-bias adjustment under the Generalized Spiked Population model}
\runtitle{PCA in high-dimensional data}

\begin{aug}
\thankstext{t1}{Supported by NIH Grants R00HL113164 and R01HG008773}
\author{\fnms{Rounak} \snm{Dey}\ead[label=e1]{deyrnk@umich.edu}\thanksref{t1}} \and
\author{\fnms{Seunggeun} \snm{Lee}\ead[label=e2]{leeshawn@umich.edu}\thanksref{t1}}

\runauthor{R. Dey \and S. Lee }

\affiliation{Department of Biostatistics, University of Michigan}

\address{Department of Biostatistics\\
	University of Michigan School of Public Health\\
1415 Washington Heights\\
Ann Arbor, MI 48109-2029, USA\\
\printead{e1}\\
\phantom{E-mail:\ }\printead*{e2}}

\end{aug}

\begin{abstract}
With the development of high-throughput technologies, principal component analysis (PCA) in the high-dimensional regime is of great interest. Most of the existing theoretical and methodological results for high-dimensional PCA are based on the spiked population model in which all the population eigenvalues are equal except for a few large ones. Due to the presence of local correlation among features, however, this assumption may not be satisfied in many real-world datasets. To address this issue, we investigated the asymptotic behaviors of PCA under the generalized spiked population model. Based on the theoretical results, we proposed a series of methods for the consistent estimation of population eigenvalues, angles between the sample and population eigenvectors, correlation coefficients between the sample and population principal component (PC) scores, and the shrinkage bias adjustment for the predicted PC scores. Using numerical experiments and real data examples from the genetics literature, we showed that our methods can greatly reduce bias and improve prediction accuracy.
\end{abstract}



\end{frontmatter}

\section{Introduction}\label{sec:intro}

Principal component analysis (PCA) is a very popular tool for analyzing high-dimensional biomedical data, where the number of features (p) is often substantially larger than the number of observations (n). PCA is widely used to adjust for population stratification in genome-wide association studies \citep{price} and to identify overall expression patterns in transcriptome analysis \citep{storey}. However, the asymptotic properties of PCA in high-dimensional data are profoundly different from the properties in low-dimensional ($p$ finite, $n\rightarrow\infty$) settings. In high-dimensional settings, the sample eigenvalues and eigenvectors are not consistent estimators of the population eigenvalues and eigenvectors \citep{johnstone2, paul}, and the predicted principal component (PC) scores based on the sample eigenvectors can be systematically biased toward zero \citep{lee1}.

There has been extensive effort to investigate the asymptotic behaviors of PCA in high-dimensional settings. To provide a statistical framework for PCA in these settings, Johnston introduced a spiked population model, which assumes that all the eigenvalues are equal except for finitely many large ones (called the spikes). A spiked population covariance matrix is basically a finite rank perturbation of a scalar multiple of the identity matrix. A typical example of a spiked population with two spikes is shown in Figure \ref{fig:Spike}. This two-spike eigenvalue structure arises if the population consists of three sub-populations, and the features are largely independent with equal variances. Under this model, convergence of sample eigenvalues, eigenvectors and PC scores have been extensively studied \citep{johnstone1, baik1, paul, lee1}.

In many biomedical data, however, the assumption of the equality of non-spiked eigenvalues can be violated due to the presence of local correlation among features. In genome-wide association studies, for example, the genetic variants are locally correlated due to linkage disequilibrium. In gene-expression data, since genes in the same pathway are often expressed together, their expression measurements are often correlated. These local correlations can cause substantial differences in non-spiked eigenvalues. To illustrate this phenomenon, we obtained eigenvalues with an autoregressive within-group correlation structure rather than the independent structure of the previous example. Figure \ref{fig:GSpike} shows that the equality assumption is clearly violated. Thus, if methods developed under the equality assumption are applied to these types of data, we will obtain biased results.

The generalized spiked population model \citep{bai5} has been proposed to address this problem. The condition that the non-spikes have to be equal is removed in this generalization. In this model the set of population eigenvalues consists of finitely many large eigenvalues called the generalized spikes, which are well separated from infinitely many small eigenvalues. Although the generalized spiked population model has a great potential to provide more accurate inference in high-dimensional biomedical data, only limited literature is available on the asymptotic properties of PCA under this model and their application to real data. \cite{bai5} and \cite{ding} provided results regarding convergence of eigenvalues and eigenvectors. However, their work remained largely theoretical.  Moreover, to the best of our knowledge, no method has been developed for estimating the correlations between the sample and population PC scores, and adjusting biases in the predicted PC scores under the generalized spiked population model.

In this paper, we systematically investigate the asymptotic behaviors of PCA under the generalized spiked population model, and develop methods to estimate the population eigenvalues and adjust for the bias in the predicted PC scores. We first propose two different approaches to consistently estimate the population eigenvalues, the angles between the sample and population eigenvectors, and the correlation coefficients between the sample and population PC scores. We compare these two methods and show the asymptotic equivalence of the estimators across them. Finally, we propose a method to reduce the bias in the predicted PC scores based on the estimated population eigenvalues.

The paper is organized as follows. We begin in Section \ref{sec:lsd} by providing the definition of the generalized spiked population model and present existing theoretical results. We develop our methods to consistently estimate the population spikes in Section \ref{sec:spike}. In Section \ref{sec:angle}, we construct consistent estimators of the angles between the sample and population eigenvectors, and the correlation coefficients between the sample and population PC scores. We also propose the bias-reduction technique for the predicted PC scores. Section \ref{sec:karoui} presents the algorithm \citep{karoui} to estimate the population limiting spectral distribution and the non-spiked eigenvalues. In Section \ref{sec:example}, we present results from simulation studies and an example from the Hapmap project to demonstrate the improved performance of our method over the existing one. Finally, we conclude the paper with a discussion.

\section{Generalized spiked population model}\label{sec:lsd}

In order to formally define generalized spiked population model, we require the concept of spectral distribution. In random matrix literature, it is natural to associate a probability measure to the set of eigenvalues as the dimension ($p$) goes to $\infty$. More explicitly, if a hermitian matrix $\Sigma_p$ has eigenvalues $\lambda_1,\lambda_2,\ldots,\lambda_p$, we can define the empirical spectral distribution (ESD) of $\Sigma_p$ to be $H_p$ based on the probability measure $$dH_p(x)=\frac{1}{p}\sum_{i=1}^{p}{\delta_{\lambda_i}(x)},$$
where $\delta_{\lambda_i}(x)$ is unity when $x=\lambda_i$, and otherwise zero. Now, for a sequence $\left \{\Sigma_p \right \}$ of covariance matrices, if the corresponding sequence $\left \{H_p \right \}$ of ESDs converge weakly to a non-random probability distribution $H$ as $p\rightarrow \infty$, then we define $H$ as the limiting spectral distribution (LSD) of the sequence $\left \{\Sigma_p \right \}$.

The generalized spiked population model \citep{bai5} is defined as follows. Suppose, $H_p$ is the ESD corresponding to the population covariance matrix $\Sigma_p$ and it converges weakly to a non-random probability distribution $H$. Let $\Gamma_H$ be the support of $H$ and $d(x,A):=\inf_{y\in A}{|x-y|}$ be the distance metric from a point $x$ to a set $A$. Then the set of eigenvalues of $\Sigma_p$ comprises of two subsets of eigenvalues where,
\begin{itemize}
	\item Non-spike: an eigenvalue $\beta$ is called a non-spike if $d(\beta,\Gamma_H)=\epsilon_p\rightarrow 0$,
	\item Generalized spike: an eigenvalue $\alpha$ is called a generalized spike if $\exists\delta > 0$ such that $\liminf d(\alpha,\Gamma_H)>\delta$, and $\alpha>\sup(\Gamma_H)$.
\end{itemize}
It is obvious from the definition that the generalized spikes are measure zero points of the population LSD. For Johnstone's spiked population model \citep{johnstone1}, the population LSD is $H=\delta_{\{1\}}$, indicating $\Gamma_H=\{1\}$. From the definition above, all eigenvalues larger than one are spikes. Hence, Johnstone's spiked population model is a special case of the generalized spiked population model.

Suppose that the population covariance matrix $\Sigma_p$ has eigenvalues $\lambda_1\geq \lambda_2\geq \ldots \geq \lambda_p$, and the sample covariance matrix $S_p=X^TX/n$ has eigenvalues $d_1\geq d_2\geq \ldots \geq d_p$, where $X$ is an $n\times p$ data matrix. Further, we will assume the following throughout the paper:
\begin{enumerate}[(A)]
	\item $p\rightarrow \infty, n\rightarrow \infty, p/n\rightarrow \gamma < \infty$.
	\item The population eigenvalues follow the generalized spiked population (GSP) model. The population ESD $H_p$ converges weakly to a non-random probability distribution $H$ with support $\Gamma_H$. Moreover, the sequence $\left \{\|\Sigma_p\|\right \}$ of spectral norms is bounded.
	\item The $n\times p$ data matrix $X=Y\Sigma_p^{1/2}$ where $Y$ is an $n\times p$ random matrix with i.i.d. elements such that $E(Y_{ij})=0,E(|Y_{ij}|^2)=1,E(|Y_{ij}|^4)<\infty$.
\end{enumerate}

From the Mar\u{c}enko--Pastur theorem \citep{mp}, the sample ESD $F_p$ converges weakly to a non-random probability distribution $F$ with support $\Gamma_F$.
For $\alpha > \sup{ \Gamma_H}$ and $x>0$, we define the following two functions
\begin{equation}\label{eqn:psif}
\psi(\alpha):=\alpha+\gamma\alpha\int{\frac{\lambda dH(\lambda)}{\alpha-\lambda}},
\quad
f_F(x):=\frac{x}{1+\gamma\int{\frac{\tau dF(\tau)}{x-\tau}}}.
\end{equation}
The following result by \cite{bai5} provides the almost sure limits of the sample eigenvalues corresponding to the population generalized spikes.
\begin{result}[\cite{bai5}]\label{thm:bai}
	Suppose assumptions (A)--(C) hold. Let $\lambda_k$ be a generalized spike of multiplicity one and the corresponding sample eigenvalue is $d_k$. Moreover, let $\psi '$ denote the first derivative of the function $\psi$. Then,
	\begin{itemize}
		\item If $\psi ' (\lambda_k)>0$, then the sample eigenvalue $d_k$ converges almost surely to $\psi (\lambda_k)$, i.e. 
		\begin{equation*}
		|d_k-\psi(\lambda_k)|\xrightarrow{a.s.} 0.
		\end{equation*}
		\item If $\psi ' (\lambda_k)\leq 0$, then let $(u_k,v_k) \subset \Gamma_H^c$ be the maximal interval on which $\psi ' >0$. Then the sample eigenvalue $d_k$ converges almost surely to $\psi(w)$ where $w$ is a boundary of $\left [ u_k,v_k\right ]$ that is nearest to $\lambda_k$. If no such interval exists, then $d_k$ converges almost surely to the $\gamma^{th}$ quantile of $F$. 
	\end{itemize}
\end{result}
Since $\psi '(\alpha)$ is a strictly increasing function for $\alpha>\sup{\Gamma_H}$, if a generalized spike $\lambda_k$ is large enough such that $\psi'(\alpha)>0$, according to Result \ref{thm:bai} the corresponding sample eigenvalue will converge almost surely to $\psi(\lambda_k)$. However if the generalized spike lies close enough, i.e. $\psi '(\lambda_k)\leq 0$, to the set of non-spikes then the convergence of the corresponding sample eigenvalue is given by the second part of the result. We will denote a generalized spike $\lambda_k$ as a ``distant spike" if $\psi ' (\lambda_k)>0$, otherwise we will call it a ``close spike".

\section{Consistent estimation of the generalized spikes}\label{sec:spike}
The following theorem provides two different consistent estimators of the distant spikes.
\begin{theorem}\label{thm:spike}
	Let $\lambda_k$ be a distant spike of multiplicity one and the corresponding sample eigenvalue is $d_k$. If the assumptions (A)--(C) hold, then,
	\begin{equation*}
	|\psi^{-1}(d_k)-\lambda_k|\xrightarrow{p} 0,
	\end{equation*}
	where $\psi^{-1}$ is the left inverse of $\psi$. Also,
	\begin{equation*} 
	|f_F(d_k) - \lambda_k|\xrightarrow{p} 0.
	\end{equation*}
	
\end{theorem}

This theorem shows that for any distant spike $\lambda_k$ we have two consistent estimators $\psi^{-1}(d_k)$ and $f_F(d_k)$. Notice that the function $f_F$ depends only on the sample LSD which can be approximated by the sample ESD. Thus, $f_F(d_k)$ can be approximated directly using the sample eigenvalues. More explicitly, if there are $m$ unique distant spikes in the population, $f_F(d_k)$ can be closely approximated as
\begin{equation*}
f_F(d_k)\approx \frac{d_k}{1+\frac{\gamma}{p-m}\sum_{i=m+1}^p{\frac{d_i}{d_k-d_i}}}.
\end{equation*}

In contrast, the $\psi$ function depends on the population LSD which is unknown. We can estimate the $\psi$ function using the algorithm described in Section \ref{sec:karoui} and then find the inverse function $\psi^{-1}$ using a Newton-Raphson type algorithm. For the close spikes, no consistent estimator exists.

\section{Consistent estimators of the asymptotic shrinkage in predicting the PC scores}\label{sec:angle}
In this section, we investigate the convergence of sample eigenvectors, PC scores, and shrinkage factors in predicting the PC scores. 
Let $e_i$ and $E_i$ to be the $i^{th}$ sample and population eigenvectors, respectively. In addition to assumptions (A)--(C), we further assume that the distant spikes are of multiplicity one. This assumption is to restrict the dimension of the corresponding eigenspaces to one, as otherwise the angle between sample and population eigenvectors, or shrinkage in predicted PC scores cannot be well defined.
\subsection{Angle between sample and population eigenvectors}
We first present the following theorem on the convergence of the quadratic forms of the sample eigenvectors.
\begin{theorem}\label{thm:angle}
	Let $\lambda_k$ be a distant spike of multiplicity one, and the assumptions (A)--(C) hold. Consider the following quadratic form
	\begin{equation*}
	\hat{\eta_k}=s_1^Te_ke_k^Ts_2,
	\end{equation*}
	where $s_1$ and $s_2$ are non-random vectors with uniformly bounded norm for all $p$. Then, 
	\begin{equation*}
	|\hat{\eta_k}-\eta_k|\xrightarrow{a.s.} 0,
	\end{equation*}
	where
	\begin{equation*}
	\eta_k=\frac{\lambda_k\psi'(\lambda_k)}{\psi(\lambda_k)}s_1^TE_kE_kTs_2
	\end{equation*}
\end{theorem}

\cite{mestre2} showed similar asymptotic properties of the quadratic forms under the assumption that the number of spikes increases with the dimension. 
Theorem \ref{thm:angle} shows the convergence of the angle between sample and population eigenvectors. 
Suppose $s_1=s_2=E_k$, and then,
\begin{equation*}
\hat{\eta_k}=E_k^Te_ke_k^TE_k=\left\langle e_k,E_k\right\rangle ^2,\quad 
\eta_k = \frac{\lambda_k\psi'(\lambda_k)}{\psi(\lambda_k)}.
\end{equation*}
Combining them, we can show
\begin{equation}
\left|\langle e_k,E_k \rangle ^2 - \frac{\lambda_k\psi'(\lambda_k)}{\psi(\lambda_k)}\right |\xrightarrow{a.s.} 0.
\end{equation}
Therefore, $\left\{ \lambda_k\psi'(\lambda_k)/\psi(\lambda_k) \right\} ^{1/2}$ is a consistent estimator of the cosine of the angle, i.e. the absolute value of the inner product, between the $k^{th}$ sample and population eigenvectors. In order to obtain this estimator we first need to estimate the $\psi$ function using the algorithm described in Section \ref{sec:karoui}.

The following result by \cite{ding} provides another consistent estimator for the angle between the $k^{th}$ sample and population eigenvectors. The proof of the asymptotic equivalence of these two estimators is given in Appendix \ref{app_proof}.
\begin{result}\label{thm:angle1}
Let $\lambda_k$ be a distant spike of multiplicity one, and $d_k$ be the corresponding sample eigenvalue. Assume that (A)--(C) hold. Define,
\begin{equation*}
g_F(x):=\left [1+\gamma f_F(x)\int{\frac{\tau dF (\tau)}{(x-\tau)^2}}\right ]^{-1}.
\end{equation*}
Then,
\begin{equation*}
|\langle e_k,E_k \rangle ^2-g_F(d_k)|\xrightarrow{p} 0.
\end{equation*}
\end{result}
Hence $g_F(d_k)^{1/2}$ also works as a consistent estimator of $|\langle e_k,E_k \rangle|$. Since the  function $g_F$ depends only on sample LSD, it can be approximated directly using  sample eigenvalues. More explicitly, if there are $m$ spikes in the population, the function $g_F$ can be closely approximated as
\begin{equation*}
g_F(d_k)\approx \left[1+\frac{\gamma f_F(d_k)}{p-m}\sum_{i=m+1}^{p}{\frac{d_i}{(d_k-d_i)^2}}\right]^{-1}.
\end{equation*}
The above equation can be used to estimate the angle between the sample and population eigenvectors.
\subsection{Correlation between sample and population PC scores}
The sample and population PC scores are the projections of the data on the sample and population eigenvectors respectively. The correlation between them can be perceived as a measure of accuracy of the PCA. The squared correlation can also be interpreted as the proportion of variance in the population PC scores that can be explained by corresponding sample PC scores. The following theorem provides the consistent estimators of the correlation between the sample and population PC scores corresponding to a distant spike.
\begin{theorem}\label{thm:score}
	Suppose $\lambda_k$ is a distant spike of multiplicity one, $d_k$ is the corresponding sample eigenvalue, and the assumptions (A)--(C) hold. Let the normalized $k^{th}$ population PC score is $P_k=XE_k / (n\lambda_k)^{1/2}$ and the normalized $k^{th}$ sample PC score is $p_k=Xe_k / (nd_k)^{1/2}$. Then,
	\begin{equation*}
	\left|\left\langle P_k,p_k \right\rangle ^2 -\psi'(\lambda_k)\right|\xrightarrow{p} 0,
	\end{equation*}
	and,
	\begin{equation*}
	\left|\left\langle P_k,p_k \right\rangle ^2-\frac{d_kg_F(d_k)}{f_F(d_k)} \right|\xrightarrow{p} 0,
	\end{equation*}	
	where the function $g_F$ is as defined in Result \ref{thm:angle1}.
\end{theorem}
Since $P_k$ and $p_k$ are normalized random vectors, the absolute value of the inner product $\langle P_k,p_k \rangle$ is identical to the absolute value of their correlation coefficient. Since correlation is scale invariant, this is also the correlation between $k^{th}$ sample and population PC scores. Therefore we can consider both ${\psi'(\lambda_k)}^{1/2}$ and $\left \{d_kg_F(d_k)/f_F(d_k) \right \}^{1/2}$ to be consistent estimators of the correlation between the $k^{th}$ sample and population PC scores.

\subsection{Asymptotic shrinkage factor}
Suppose $\lambda_k$ is a distant spike. Let the $k^{th}$ sample PC score for the $j^{th}$ observation $x_j$ be $p_{kj}=x_j^Te_k$, and the $k^{th}$ predicted PC score for a new observation $x_{new}$ be $q_k=x_{new}^Te_k$. Then the quantity $\rho_k=\lim_{p\rightarrow\infty}{\{E(q_k^2)/E(p_{kj}^2)\}^{1/2}}$ describes the asymptotic shrinkage in the $k^{th}$ predicted PC score for a new observation. As both $p_{kj}$ and $q_k$ are centered, i.e. $E(p_{kj})=E(q_k)=0$, $\rho_k$ represents the limiting ratio of the standard deviations of the predicted PC scores and the sample PC scores. Therefore, if we can estimate $\rho_k$, then the shrinkage bias in the $k^{th}$ predicted PC scores can be easily adjusted by rescaling the predicted scores by the factor $\rho^{-1}_k$. The following theorem provides the consistent estimator of the asymptotic shrinkage factor $\rho_k$.
\begin{theorem}\label{thm:shrinkage}
	Suppose $\lambda_k$ is a distant spike of multiplicity one, $d_k$ is the corresponding sample eigenvalue, and the assumptions (A)--(C) hold. Let $p_{kj}$ and $q_k$ be as defined above. Then,
	\begin{equation*}
	\left |\sqrt{\frac{E(q_k^2)}{E(p_{kj}^2)}} - \frac{\lambda_k}{d_k}\right | \xrightarrow{p}0.
	\end{equation*}
\end{theorem}

This is a surprising result in which the asymptotic shrinkage factor is expressed as a simple ratio of the population and sample eigenvalues. 
Recall that we already constructed the consistent estimators for population eigenvalues in the previous sections. Using these results, 
the asymptotic shrinkage factor $\rho_k$ can be consistently estimated by $\hat{\lambda}_k/d_k$ where $\hat{\lambda}_k$ is any consistent estimator of $\lambda_k$.

\subsection{Comparison of the two different estimators}
For each of the quantities discussed above, we proposed two asymptotically equivalent estimators. In terms of practical applications they have their own advantages and disadvantages. One of them can be approximated directly based only on the sample eigenvalues, while the other one requires to estimate the LSD of the population eigenvalues to obtain the $\psi$ function.
For ease of discourse we will call the former ``$d$-estimator'' and the later ``$\lambda$-estimator''. If the number of spikes is known, estimating the $d$-estimator is computationally more efficient than estimating the $\lambda$-estimator as it does not involve estimating the population LSD. However, by estimating the population LSD the $\lambda$-estimation procedure can verify whether an estimated eigenvalue is actually a distant spike by checking if $\psi '>0$. Thus it can be used to estimate the number of distant spikes when it is unknown (see Section \ref{sec:karoui}). On the other hand, the $d$-estimation procedure provides no information on the population LSD and thus cannot distinguish among distant spikes, close spikes and non-spikes. To summarize, when the number of spikes is known or we only want to estimate few of the largest eigenvalues which are known to be distant spikes, then the $d$-estimation procedure has the advantage of a faster computation, while the $\lambda$-estimation procedure is more useful when the number of spikes is unknown or the distribution of the non-spikes is of interest.

\subsection{Comparison of Generalized Spiked Population (GSP) model and Spiked  Population (SP) model}
As mentioned before, the SP model \citep{johnstone1} is a special case of the GSP model. It is easy to verify that when the population eigenvalues follow the SP model, our consistent estimators for the spiked eigenvalues, the angles between the eigenvectors, the correlation coefficients between the PC scores and the shrinkage factors conform to the consistent estimators derived by \cite{baik1, paul, lee1}. For an SP model where all the non-spikes are equal to one, the LSD $H$ is a degenerate distribution at one, and
\begin{equation*}
\psi(\alpha)=\alpha\left (1+\frac{\gamma}{\alpha-1} \right ); \quad
\psi'(\alpha)=1-\frac{\gamma}{(\alpha-1)^2}.
\end{equation*}
Now, $\psi'(\alpha)>0$ if and only if $\alpha>1+\gamma^{1/2}$. If $\alpha>1+\gamma^{1/2}$ and $d$ is the corresponding sample eigenvalue, then the consistent estimator of $\alpha$ is given by $\psi^{-1}(d)$, and
$$\frac{\alpha\psi '(\alpha)}{\psi(\alpha)}=\frac{1-\frac{\gamma}{(\alpha-1)^2}}{1+\frac{\gamma}{\alpha-1}};\quad \frac{\alpha}{\psi(\alpha)}=\frac{\alpha-1}{\alpha+\gamma-1},$$ which show that all our results match with the results from \cite{lee1}.

It is of interest to investigate how closely methods developed under the SP model can approximate the consistent estimators for the distant spikes when the population eigenvalues actually follow a GSP model. 
Suppose the population eigenvalues $\lambda_1\geq\lambda_2\geq \ldots \geq\lambda_p$ follow the GSP model with $m$ distant spikes. The sample eigenvalues are $d_1\geq d_2\geq \ldots \geq d_p$. Let $\lambda_k$ be a distant spike with multiplicity one, and the corresponding sample eigenvalue is $d_k$. Then according to Result \ref{thm:bai}, $d_k\rightarrow\psi(\lambda_k)$ almost surely. From the definition of $\psi$,
\begin{equation*}
\psi(\lambda_k)=\lambda_k \left (1+\gamma\int{\frac{\lambda dH(\lambda)}{\lambda_k-\lambda}} \right)=\lambda_k +\gamma\int{\frac{\lambda dH(\lambda)}{1-\lambda/\lambda_k}}.
\end{equation*}
If $H$ is almost degenerate, i.e., the non-spikes are nearly identical, then,
\begin{equation}\label{eqn:gspsp1}
\psi(\lambda_k)\approx\lambda_k+\frac{\gamma\bar{\lambda}}{1-\bar{\lambda}/\lambda_k},
\end{equation}
where $\bar{\lambda}=\int{\lambda dH(\lambda)}$ is the mean of the population LSD which can be closely approximated by the mean of the non-spikes.
On the other hand, if the spike $\lambda_k$ is very large compared to all the non-spikes such that $\lambda/\lambda_k\approx 0$ for any $\lambda\in \Gamma_H$, then
\begin{equation}\label{eqn:gspsp2}
\psi(\lambda_k)\approx\lambda_k+\gamma\bar{\lambda}.
\end{equation}

Now, suppose instead of using the GSP assumption, we use the SP assumption to estimate the distant spikes. We assume that under the SP model the population covariance matrix is scaled by a factor $\zeta$ and the population eigenvalues are $\beta_1\geq\beta_2\geq \ldots \geq\beta_m>\zeta=\zeta=\ldots\zeta$. If $\beta_k$ is the population eigenvalue corresponding to $d_k$, then $d_k\rightarrow\psi(\beta_k)$ almost surely where,
\begin{equation*}
\begin{split}
\psi(\beta_k)&=\beta_k\left(1+\gamma\frac{\zeta}{\beta_k-\zeta}   \right)\\
&=\beta_k+\frac{\gamma\zeta}{1-\zeta/\beta_k}.
\end{split}
\end{equation*}
Here $\zeta$ is estimated as the mean of the non-spikes as they are all assumed to be equal to $\zeta$. Notice that this expression is approximately equal to the expression in \eqref{eqn:gspsp1} with $\beta_k=\lambda_k$ and $\zeta=\bar{\lambda}$. Therefore, the asymptotic limit of $d_k$ under both the GSP and the SP model are approximately equal when the non-spikes are nearly identical. On the other hand, when the spike $\beta_k$ is very large compared to all the non-spikes such that $\zeta/\beta_k\approx 0$, then $\psi(\beta_k)\approx \beta_k + \gamma\zeta$. In this case also, the asymptotic limit of $d_k$ under both the GSP and the SP model are approximately equal with $\beta_k=\lambda_k$ and $\zeta=\bar{\lambda}$. Therefore if a generalized spike is very far away from the support of the population LSD, then the estimate of the spike based on an SP model will closely approximate the estimate based on a GSP model. However the SP model will provide potentially biased estimates if the non-spikes are not similar and the ratio between the largest non-spike and the spike of interest is substantially larger than zero.

\section{Estimation of the population LSD}\label{sec:karoui}
The $\lambda$-estimators rely on $\psi$, that is a function of the unknown population LSD $H$. To use the  $\lambda$-estimators, it is thus required to estimate $H$.
Using the Stieltjes transformation and the Mar\u{c}enko--Pastur theorem, \cite{karoui} developed a general algorithm to estimate the population LSD from the sample ESD, $F_p$. 
We propose to use Karoui's method to estimate the population LSD $H$ and then use it to estimate  $\psi$. 

\subsection{Karoui's algorithm}
Suppose $v_{F_p}$ is the Stieltjes transformation of the set of eigenvalues in the sample covariance matrix in which 
\begin{equation*}
v_{F_p}(z)=\frac{1}{n}\sum_{i=1}^{n}{\frac{1}{d_i-z}}
\end{equation*}
for any $z\in \mathbb{C^+}, \mathbb{C^+}=\{x\in \mathbb{C},Im(x)>0\}$.
According to the Mar\u{c}enko--Pastur theorem \citep{mp}, when assumptions (A)--(C) hold, $v_{F_p}$ converges pointwise almost surely to a non-random limit $v_F$, which uniquely satisfies the following equation
\begin{equation*}
v_{F}(z)=-\left (z-\gamma\int{\frac{\lambda dH(\lambda)}{1+\lambda v_F(z)}} \right)^{-1}.
\end{equation*}
Karoui's method first calculates $v_{F_p}$ for a grid of values $\left \{z_j \right \}_{j=1}^J$, and then finds $\hat{H}$ as a solution to minimize the following objective function
\begin{equation*}
\hat{H}=\arg_H \min L\left (\left\{ \frac{1}{v_{F_p}(z_j)}+z_j-\frac{p}{n}\int{\frac{\lambda dH(\lambda)}{1+\lambda v_{F_p}(z_j)}}  \right\}_{j=1}^J \right ),
\end{equation*}
where $L$ is any pre-defined convex loss function.
In order to approximate the integral inside of the loss function, the algorithm discretizes $H$ in the following way,
\begin{equation*}
dH(\lambda)\simeq\sum_{k=1}^K{w_k\delta_{t_k}(\lambda)},
\end{equation*}
where $\delta_{t_k}(\lambda)=1$ if $\lambda=t_k$ and $0$ otherwise, $\sum_{k=1}^K{w_k}=1$ with $w_k>0$ for all $k$, and $\{t_k \}_{k=1}^{K}$ is a grid of points on the support of $H$. This is basically approximating $H$ by a discrete distribution with support $\{t_k \}_{k=1}^{K}$. Then the integral is approximated by
\begin{equation*}
\int{\frac{\lambda dH(\lambda)}{1+\lambda v_F(z)}} \simeq \sum_{k=1}^K{w_k\frac{t_k}{1+t_kv_{F_p}(z_j)}},
\end{equation*}
and the minimization problem transforms into,
\begin{equation}\label{eqn:EstH}
\hat{H}=\arg_H \min L\left (\left\{ \frac{1}{v_{F_p}(z_j)}+z_j-\frac{p}{n}\sum_{k=1}^K{w_k\frac{t_k}{1+t_kv_{F_p}(z_j)}}  \right\}_{j=1}^J \right ).
\end{equation}
\cite{karoui} has shown the weak convergence of $\hat{H}$ to $H$, i.e $\hat{H}\rightarrow H$.

Some examples of the convex loss function $L$ can be,
\begin{itemize}
	\item $L_\infty \left( \{e_j\}_{j=1}^J\right)=\max_{j} \max \left\{|Re(e_j)|,|Im(e_j)| \right\}$
	\item $L_1\left( \{e_j\}_{j=1}^J \right)=\sum_{j=1}^J{|e_j|}$
	\item $L_2\left( \{e_j\}_{j=1}^J \right)=\sum_{j=1}^J{|e_j|^2}$
\end{itemize}
For the convex loss functions described above, the estimation of $H$ in \eqref{eqn:EstH} reduces to a convex optimization problem \citep{conv}. Karoui also provided a translation of this problem into a linear programming problem when $L_\infty$ loss function is used. Further details can be found in \cite{karoui}.

\subsection{Implementing Karoui's algorithm when the number of spikes is known}
Since the generalized spikes fall outside the support of the population LSD, Karoui's algorithm cannot be directly applied to estimate the spikes. Furthermore, \cite{bai2} showed that the probability of a sample eigenvalue falling outside the support of the sample LSD will go to zero as $p$ increases, which implies that the sample eigenvalues corresponding to the population generalized spikes will be measure zero points in the sample LSD. Since the spikes behave like measure zero points (or outliers) when we are concerned about estimating the population LSD, we can exclude the sample eigenvalues corresponding to the population generalized spikes while calculating $v_{F_p}$ and that will lead to a more robust estimation of $H$. Therefore, we will apply Karoui's algorithm in the following way,
\begin{enumerate}[(I)]
	\item Suppose the population covariance matrix possesses $m$ generalized spikes. We exclude the top $m$ sample eigenvalues while calculating $v_{F_p}$,
	\begin{equation*}
	v_{F_p}(z)=\frac{1}{n-m}\sum_{i=1}^n{\frac{1}{d_i-z}}.
	\end{equation*} 
	\item Apply Karoui's algorithm to obtain $\hat{H}$. Since it is reasonable to assume that the true population LSD is a continuous or piecewise continuous distribution function, kernel smoothing algorithm can be used on $\hat{H}$ to obtain a more continuous approximation of $H$.
	\item The quantiles of $\hat{H}$ can be considered as the estimators of the non-spikes.
	\item Suppose, $\hat{\lambda}_{m+1},\hat{\lambda}_{m+2},\ldots,\hat{\lambda}_p$ are the estimated non-spikes. Then the $\psi$ function is estimated by,
	\begin{equation*}
	\hat{\psi}(\alpha)=\alpha+\frac{\gamma\alpha}{p-m}\sum_{i=m+1}^p{\frac{\hat{\lambda}_i}{\alpha-\hat{\lambda}_i}}.
	\end{equation*}
\end{enumerate}
Due to the weak convergence $\hat{H}\rightarrow H$, $\hat{\psi}$ will also converge to $\psi$ point-wise. Thus, all the estimates provided in Section \ref{sec:spike} and \ref{sec:angle} will still be consistent if we replace $\psi$ with $\hat{\psi}$.

\subsection{Estimating the number of spikes}\label{sub:estSpike}
Our application of Karoui's algorithm to the GSP model depends on the number of spikes $m$, which is usually unknown. If we have some knowledge of the underlying structure of the data, we can use it to estimate $m$ roughly. Suppose we know that the data are coming from a mixture of $K$ subpopulations, and within each subpopulation the observations are i.i.d.. Since the spikes represent the between group differences, the number of spikes should be the same as the rank of the between group covariance matrix which is $(K-1)$. However in real data, it is often hard to accurately assess the number of such homogeneous subpopulations. In those cases we can use the following algorithm to estimate $m$.
\begin{enumerate}[(I)]
	\item Start with a reasonable finite upper bound $m_{max}$ of the number of spikes. The upper bound can be selected based on prior information on the subpopulations, or by examining the sample eigenvalues. Set $m=m_{max}$.
	\item Use Karoui's algorithm to estimate the population LSD and the non-spikes. Suppose the estimated non-spikes are $\hat{\lambda}_{m+1}\geq\hat{\lambda}_{m+2}\geq\ldots\geq\hat{\lambda}_p$, and the $\psi$ function is estimated by,
	\begin{equation*}
	\hat{\psi}(\alpha)=\alpha+\frac{\gamma\alpha}{p-m}\sum_{i=m+1}^p{\frac{\hat{\lambda}_i}{\alpha-\hat{\lambda}_i}}.
	\end{equation*}
	\item Find $S_\psi>\lambda_{m+1}$ using Newton-Raphson algorithm such that $$\hat{\psi}'(S_\psi)=1-\frac{\gamma}{p-m}\sum_{i=m+1}^p{\left (\frac{\hat{\lambda}_i}{S_\psi-\hat{\lambda}_i}\right )^2}=0.$$ 
	\item Since any distant spike must be larger than $S_\psi$, and $\hat{\psi},\hat{\psi}'$ are both continuous and strictly increasing functions on $(S_\psi,\infty)$, the equation $\hat{\psi}(\lambda)-d_k=0$ has a root in $(S_\psi,\infty)$ if and only if $\hat{\psi}(S_\psi)-d_k<0$. Therefore, find the smallest index $i^*$ in $1,2,\ldots,m$ such that $d_{i^*}\leq \hat{\psi}(S_\psi)$. If all $d_1,d_2,\ldots,d_m$ are larger than $\hat{\psi}(S_\psi)$ then stop and select $m$ as the number of distant spikes. Otherwise, set $m=i^*-1$ and repeat step (II)-(IV).
\end{enumerate}
Note that the close spikes occur so close to the support of the population LSD that they cannot be distinguished separately from the non-spikes when the number of spikes is unknown.

\section{Simulation studies and real data example}\label{sec:example}
\subsection{Simulation studies}
In this section we will present simulation studies of five different scenarios to compare the performances of the proposed GSP-based methods and the existing SP-based method proposed by \cite{lee1}. For each of the first four studies, we simulated a training dataset with $n=500$ individuals and $p=5000$ features. The data were generated from three subpopulations with sample sizes 100, 150 and 250. For each subpopulation we first selected a mean vector $\mu_i$ by drawing its elements randomly with replacement from $\{-0.3,0,0.3\}$. Then samples in the $i^{th}$ subpopulation were drawn from $N_p(\mu_i,V)$ where $V$ is the AR(1) covariance matrix with variance $\sigma^2$ and autocorrelation $\rho$. The $(\sigma^2,\rho)$ pairs used for the four studies were $(4,0.8),(1,0.7),(7.5,0.8)$ and $(4,0)$. In study 5, we considered an ultra high-dimensional setting where $n=500$ samples and $p=50000$ features were simulated from a population with two spikes at $300$ and $250$, and the non-spikes generated from the AR(1) covariance structure with $(\sigma^2,\rho) = (4, 0.8)$. The population eigenvalue plots for all the studies are shown in Figure \ref{fig:sim}.

We also generated test datasets for each study with the same settings as the training datasets. Then we applied our GSP-based methods and the existing SP-based method to estimate the population spikes, the angles between the sample and population eigenvectors, the correlations between the sample and population PC scores and the asymptotic shrinkage factors. For all of the studies, we used the upper bound $m_{max}=5$ to estimate the number of distant spikes using the algorithm described in \ref{sub:estSpike}. We simulated each study 200 times to calculate the empirical biases and standard errors of the estimates. The results are presented in Table \ref{tab:sim}.

It is clear from Table \ref{tab:sim} that for Study 1, 2, 3 and 5 our methods reduced the bias in all the estimates while having similar standard errors as the existing method. The positive empirical biases in all the SP estimates suggest that the SP method tends to overestimate all the quantities. In Study 4, since the underlying population satisfied the SP assumption, all methods provided very similar and almost unbiased estimates ($<1\%$). The results also verify that the $\lambda$-estimates and $d$-estimates are asymptotically equivalent. The performances of the $\lambda$-estimates and the $d$-estimates are nearly identical in all the simulation studies.

In Study 1, the ratio of the largest non-spike with the two spikes are 0.29 and 0.48, which are substantially larger than zero. Thus according to the discussion in Section \ref{sec:angle} the SP model does not closely approximate the GSP model. The results support this assertion as the SP model-based estimates are highly biased whereas the estimates based on our methods have very little empirical bias. On the other hand, in Study 2 the largest non-spike is very small compared to the smallest spike (ratio 0.08). Thus the estimates based on the SP model closely approximate the estimates based on the GSP model, and we find very little empirical bias ($<1\%$) in all of the SP model-based estimates. In Study 3, even though there were two spikes present, only the largest population eigenvalue was a distant spike. So we presented only the estimates corresponding to the largest population eigenvalue. Since the ratio of the largest non-spike and the largest spike is substantially larger than zero (0.53) in this study, we observe very high empirical bias in the SP model-based estimates. However, our methods provided negligible empirical biases even in the presence of a close spike. We also presented the estimated number of distant spikes in each of the simulation studies in \ref{suppA}. Note that in some cases our algorithm over-estimates the number of distant spikes. However, as the over-estimation is finite, the estimates of the distant spikes still remain consistent.

Study 5 shows the performances of our methods when $p>>n$. In this setting,  $p/n=100$ is comparable to $n=500$. Therefore this asymptotic setting falls into the category of ultra-high dimensional data where $p/n\rightarrow \infty$. 
Although our asymptotic results were derived based on the assumption that $p/n\rightarrow \gamma<\infty$, the results from study 5 show that even for this ultra-high dimensional setting our methods provided almost unbiased estimates.

\subsection{Application on Hapmap III data}

For this demonstration we used genetic data from the Hapmap Phase III project (\url{http://hapmap.ncbi.nlm.nih.gov/}). Our sample consisted of unrelated individuals sampled from two different populations: a) Utah residents with Northern and Western European ancestry (CEU) and b) Toscans in Italy (TSI). We only included genomic markers that are on chromosome 1-22, have less than $5\%$ missing values, and those with minor allele frequency more than 0.05. We also excluded 2 samples (both from CEU) with outlier PC scores (more than six standard deviations away from the mean PC score corresponding to at least one distant spike). We then mean-centered and variance-standardized the data for each marker. The final sample consisted of 198 individuals (110 from CEU and 88 from TSI). Total number of markers selected across chromosome 1-22 was 1389511.

To evaluate the performance of the proposed methods with different $p$, we  performed PCA on each chromosome separately. The number of markers varied from 19331 (chromosome 21) to 116582 (chromosome 2). The distribution of the number of markers across different chromosomes are presented in \ref{suppA}.
We first estimated the number of distant spikes using the algorithm described in Section \ref{sub:estSpike}. We found no distant spike in chromosome $22$ and only one distant spike in chromosome $2$. Then we applied our GSP-based methods and the existing SP-based method \citep{lee1} to estimate the asymptotic shrinkage factors corresponding to the distant spikes. Figure \ref{fig:PC_hapmap1}, \ref{fig:PC_hapmap2} compares the estimated asymptotic shrinkage factors for the first two PCs across different chromosomes.
The plots show that for all the chromosomes, $\lambda$-GSP and $d$-GSP methods provided almost equal estimates while the SP estimates are larger than both the GSP estimates. This suggests that the SP method would over-estimate the shrinkage factors when the population eigenvalues deviate from the assumption that the non-spiked eigenvalues are the same.

To investigate whether the proposed shrinkage-bias adjustment can improve the prediction accuracy, we performed a leave-one-out cross-validation. 
In each iteration we removed one individual (test sample) and performed PCA on the remaining individuals (training samples) to predict the PC score of the test sample.
For each predicted PC score, we adjusted the shrinkage-bias using the GSP-based and SP-based shrinkage factor estimates.
One important issue with this cross-validation is that the exclusion of one individual can substantially change the PC-coordinates, in which 
the PC score plots from the training sample-based and complete sample-based PCA can be substantially different. 
In order to circumvent this problem, in each iteration we first rescaled the PC scores based on their corresponding sample eigenvalues to make the PCs comparable. 
In addition, we obtained the mean squared difference of the training sample PC1-2 scores with and without the exclusion of the test sample (for chromosome 2, only PC1 is used), and excluded the test sample from the prediction error estimation if the mean squared difference was above a threshold $\epsilon$.
We used four different values $0.5,1,5$ and $10$ for the threshold parameter $\epsilon$, and for each value of $\epsilon$ we calculated the mean squared errors (MSE) of the unadjusted and adjusted (using the GSP and SP-based methods) predicted PC scores of the test samples. The sample sizes of the test samples that were finally included in the prediction error estimation for different values of $\epsilon$ are shown in \ref{suppA}.
Figure \ref{fig:error_ratio} shows the estimated MSEs for $\epsilon=1$. It is clear that both the $\lambda$-GSP and $d$-GSP methods have much smaller MSEs than the SP method. As expected, the unadjusted predicted PC scores have substantially larger MSE than all the proposed adjustments. The plots are very similar for the other values of $\epsilon$, and they can be found in \ref{suppA}.

Figure \ref{fig:bias_hapmap} illustrates the shrinkage-bias adjustment for the PC1 and PC2 scores of an individual based on the markers on chromosome 7. The plot clearly shows that the bias-adjusted PC score based on the SP model is still biased towards zero, whereas the bias-adjusted PC score based on the GSP model is very close to the original sample PC score. We only showed the $d$-GSP adjusted score in the plot as the $d$-GSP and $\lambda$-GSP adjusted scores were almost equal.

\section{Conclusions and discussion}\label{sec:conclusions}
In this paper, we investigated the asymptotic properties of PCA under the Generalized Spiked Population model and derived estimators of the population eigenvalues, the angles between the sample and population eigenvectors, and the correlation coefficients between the sample and population PC scores. We also proposed methods to adjust the shrinkage bias in the predicted PC scores.
Since the proposed methods do not require the equality of the non-spiked eigenvalues, they can be widely used in high-dimensional biomedical data analysis.
We also implemented all our algorithms in the R package \textbf{hdpca}.

We note that \cite{mestre1,mestre2} proposed an asymptotic setting similar to the generalized spiked population model but with a different assumption on the number of spikes in which the number of spikes increases with the dimension. Under this assumption, he provided asymptotic properties of sample eigenvalues and eigenvectors. However, in many biomedical data, the number of spikes is usually finite as the spikes represent the difference between finitely many underlying subpopulations. Therefore we believe that the generalized spiked population model is more appropriate in such cases.

Since our results were derived under the assumption that $p$ and $n$ grow at the same rate, one open question is whether the method can be applied for the ultra high-dimensional data where $p$ is greatly larger than $n$. In our simulation studies, however, we showed that our method can provide almost unbiased estimators when $p/n$ is as large as $100$. We note that, there exist asymptotic regimes which are suitable for the ultra high-dimensional data. For example, the high-dimension low sample size regime, in which $p \rightarrow \infty$ but $n$ is finite, was discussed in \cite{hall,ahn, jung}. \cite{lee2} considered an alternative ultra high-dimensional regime, where $p$ and $n$ grow to infinity with $p / n \rightarrow \infty$. They also showed that methods developed under the assumption that $p/n$ is finite can be applied to the ultra-high dimensional data. In future research, we will use this asymptotic regime to investigate behaviors of our methods in the ultra-high dimensional regime.

In some special cases, even though the features exhibit strong local correlation, one can use the spiked population model based methods after some suitable data manipulation. In genome-wide association studies, SNP pruning \citep{prune} can be used to remove locally correlated SNPs to satisfy the spiked population model. For example, \citet{lee1} reported good performance of the spiked population model-based methods with the SNP-pruned Hapmap III dataset. This approach, however, can lead to a considerable loss of information; the SNP-pruning in Hapmap III data removed nearly $90\%$ of the SNPs. Since the proposed approach does not require this additional step, it can use most of the information present in the data.

\appendix

\section{Proofs}\label{app_proof}

\begin{proof}[\textbf{Proof of Theorem \ref{thm:spike}}]
The first part of the proof follows directly from Result \ref{thm:bai} along with the fact that on the domain of the distant spikes, the $\psi$ function is strictly increasing, and hence is left invertible. Since $\psi ''(\alpha)>0$ for any $\alpha>\sup{\Gamma_H}$, $\psi'(\alpha)$ is a strictly increasing function for $\alpha>\sup{\Gamma_H}$. Let $S_\psi >\sup{\Gamma_H}$ be a solution for $\psi'(\alpha)=0$. Then for any  $\alpha>\sup{\Gamma_H}$, $\psi'(\alpha)>0$ if and only if $\alpha>S_\psi$. 
Therefore the interval $(S_\psi,\infty)$ is the domain of the distance spikes, and $\psi$ is a strictly increasing function on this interval. The second part follows from Lemma \ref{lemma:bridge}.
\end{proof}

\begin{proof}[\textbf{Proof of Theorem \ref{thm:angle}}]
The proof closely follows the proof of Theorem 2 in \cite{mestre1}. However, contrary to \cite{mestre1}, we do not assume that the population LSD contains the generalized spikes. Thus, some of the derivation steps and results are substantially different from \cite{mestre1}. We start the derivation by first noting that the quadratic forms $\hat{\eta}_k$ can be expressed as contour integrals of a special class of Stieltjes transforms of the sample covariance matrix. Let us define,
$$\hat{m}_p(z) := s_1^T\left (S_p- zI_p\right )^{-1} s_2=\sum_{j=1}^{p}{\frac{s_1^Te_je_j^Ts_2}{d_j-z}};\quad \forall z\in \mathbb{C^+}$$
where $s_1$ and $s_2$ are non-random vectors with uniformly bounded norms. \cite{girko} and \cite{mestre3} showed that under the assumption that the population LSD contains the generalized spikes, 
\begin{equation}\label{eqn:girko}
\left |\hat{m}_p(z)-m_p(z) \right |\xrightarrow{a.s.}0;\quad \forall z\in\mathbb{C^+}
\end{equation}
where
$$m_p(z)=s_1^T\left [w(z)\Sigma_p-zI_p\right ]^{-1}s_2=\sum_{j=1}^{p}{\frac{s_1^TE_jE_j^Ts_2}{w(z)\lambda_j-z}}. $$
The function $w(z)$ is defined as $w(z)=1-\gamma-\gamma z b_F(z)$ where
$b_F(z)=\int{(\tau-z)^{-1}dF(\tau)}$ is the Stieltjes transform of the sample LSD. It is easy to check, by the same arguments provided in \cite{mestre3}, that the result still holds when the generalized spikes are considered lying outside the support of the population LSD.
The functions $\hat{m}_p,m_p$ and $b_F$ can be extended to $\mathbb{C^-}=\{z\in\mathbb{C}:Im(z)<0\}$ by defining $\hat{m}_p(z)=\hat{m}_p^*(z^*),m_p(z)=m_p^*(z^*)$ and $b_F(z)=b_F^*(z^*)$ for $z\in \mathbb{C^-}$ where $z^*$ is the complex conjugate of $z$. With this definition, $\left |\hat{m}_p(z)-m_p(z) \right |\xrightarrow{a.s.}0 $ even when $z\in \mathbb{C^-}$. 
Now $\hat{\eta}_k$ can be expressed as an integral of $\hat{m}_p$,
\begin{equation*}
\hat{\eta}_k=\frac{1}{2\pi i}\oint_{\partial\hat{\mathbb{R}}_y^{-}(k)}{\hat{m}_p(z)dz},
\end{equation*}
where $i=\sqrt{-1},y>0$ and $\partial\hat{\mathbb{R}}_y^{-}(k)$ is the negatively (clockwise) oriented boundary of the rectangle $\hat{\mathbb{R}}_y(k)=\left \{z\in \mathbb{C}:\hat{a}_1\leq Re(z)\leq \hat{a}_2, |Im(z)|\leq y \right\}$. $\hat{a}_1$ and $\hat{a}_2$ can be arbitrarily chosen provided that $\hat{\mathbb{R}}_y(k)$ contains only the sample eigenvalue $d_k$ and no other sample eigenvalue. Then the following lemma gives the almost sure limit of $\hat{\eta}_k$.
\begin{lemma}\label{lemma:dct}
\begin{equation*}
\left | \frac{1}{2\pi i}\oint_{\partial\hat{\mathbb{R}}_y^{-}(k)}{\hat{m}_p(z)dz}- \frac{1}{2\pi i}\oint_{\partial\mathbb{R}_y^{-}(k)}{m_p(z)dz} \right |\xrightarrow{a.s.}0,
\end{equation*}
where $y>0$ and $\partial\mathbb{R}_y^{-}(k)$ is the negatively (clockwise) oriented boundary of the rectangle $\mathbb{R}_y(k)=\left \{z\in \mathbb{C}:a_1\leq Re(z)\leq a_2, |Im(z)|\leq y \right\}$. $a_1$ and $a_2$ can be arbitrarily chosen so that $\psi(\lambda_k)\in[a_1,a_2]$ and $[a_1,a_2]\subset \psi\left (S_\psi,\infty\right )$ where $S_\psi>\sup{\Gamma_H},\psi ' (S_\psi)=0$. $\psi\left (S_\psi,\infty\right )$ denotes the image of the interval $(S_\psi,\infty)$ under $\psi$.
\end{lemma}
Lemma \ref{lemma:dct} implies
\begin{equation}\label{eqn:eta2}
\left | \hat{\eta_k}- \sum_{j=1}^{p}{\left (\frac{1}{2\pi i} \oint_{\partial\mathbb{R}_y^{-}(k)}{\frac{dz}{w(z)\lambda_j-z}}	\right )s_1^TE_jE_j^Ts_2} \right |\xrightarrow{a.s.}0.
\end{equation}
Now we need to evaluate the integral in \eqref{eqn:eta2} in order to get the almost sure limit of the random variable $\hat{\eta}_k$. First, we extend the $\psi$ function to $\mathbb{R}_y(k)$ as follows,
$$\psi(z):=z\left ( 1+\gamma\int{\frac{\lambda dH(\lambda)}{z-\lambda}}\right ), \quad \forall z \in \mathbb{R}_y(k).$$
According to \cite{mp}, for all $z\in\mathbb{C^+}$, $b_F(z)=b$ is the unique solution to the following equation
\begin{equation}\label{eqn:mpangle}
b=\int{\frac{dH(\lambda)}{\lambda(1-\gamma-\gamma z b)-z}}
\end{equation} in the set $\left \{b\in \mathbb{C}:\gamma b-(1-\gamma)/z\in \mathbb{C^+} \right \}$.
It is easy to see that $b_F$  also satisfies \eqref{eqn:mpangle} when $z \in \mathbb{C^-}$. Now we formally define the $f_F$ function introduced in \eqref{eqn:psif},
\begin{equation}\label{eqn:definef}
f_F(z):=\frac{z}{w(z)}=\frac{z}{1-\gamma-\gamma z b_F(z)}, \quad \forall z\in \mathbb{C\setminus R}.
\end{equation} 
Then $b_F$ can be expressed in terms of $f_F$ as, $$b_F(z)=\frac{(1-\gamma)f_F(z)-z}{\gamma z f_F(z)}.$$
By replacing $b$ with $\left [(1-\gamma)f-z \right ]/\gamma z f$ in \eqref{eqn:mpangle},
\begin{equation}\label{eqn:mp2angle}
f\left (1+\gamma\int{\frac{\lambda dH(\lambda)}{f-\lambda}}\right )=z.
\end{equation}
It is easy to see that $b_F$ is a solution to \eqref{eqn:mpangle} if and only if $f_F$ is a solution to \eqref{eqn:mp2angle}. Therefore, for all $z\in \mathbb{C^+}$ (similarly for $z\in\mathbb{C^-}$), $f_F(z)=f$ is the unique solution to \eqref{eqn:mp2angle} on $\mathbb{C^+}$ (respectively, $\mathbb{C^-}$). This implies $\psi\left (f_F(z)\right )=z$ for all $z\in\mathbb{R}_y(k)\setminus [a_1,a_2]$.

Now we focus on the case when $z\in \mathbb{R\setminus}\{0\}$. According to \cite{silv2}, we can extend $b_F$ to $\mathbb{R\setminus} \{0\}$ by defining $b_F(z)=\lim_{y\rightarrow 0^+}{b_F(z+iy)}$ for any $z\in \mathbb{R\setminus}\{0\}$. The definition of $f_F$ can also be extended in a similar fashion. 
In Lemma \ref{lemma:bridge} we have shown that $f_F$ is the inverse function of $\psi$ on $(S_\psi,\infty)$, and there exists $M_f>\sup{\Gamma_F}$ for which $\psi\left (S_\psi,\infty\right )=(M_f,\infty)$. Thus, $[a_1,a_2]\subset\psi\left (S_\psi,\infty\right )$ implies $\psi\left (f_F(z)\right )=z$ for all $z\in\mathbb{R}_y(k)$. Furthermore, the function $\psi$ is continuous and differentiable on $\mathbb{R}_y(k)$, and the derivative is given by, $$\psi ' (z)=1-\gamma\int{\left(\frac{\lambda}{z-\lambda}\right)^2dH(\lambda)}.$$ Then the integral in \eqref{eqn:eta2} can be expressed in terms of $\psi$ and $f_F$ as follows,
\begin{equation}\label{eqn:integral}
\begin{split}
\frac{1}{2\pi i} \oint_{\partial\mathbb{R}_y^{-}(k)}{\frac{dz}{w(z)\lambda_j-z}}	&=\frac{1}{2\pi i} \oint_{\partial\mathbb{R}_y^{-}(k)}{\frac{dz}{\frac{z}{f_F(z)}\lambda_j-z}}\\
&=\frac{1}{2\pi i} \oint_{\partial\mathbb{R}_y^{-}(k)}{\frac{1}{\lambda_j-f_F(z) }.\frac{f_F(z)}{\psi\left (f_F(z)\right )}dz }.
\end{split}
\end{equation}
The integrand in the final expression is holomorphic on $\mathbb{R}_y^{-}(k)$ when $j\neq k$ and possesses a simple pole $\psi(\lambda_k)$ when $j=k$. Therefore, when $j\neq k$ the integral in \eqref{eqn:integral} is zero. When $j=k$, Applying the residue theorem on the final integral,
\begin{equation*}
\begin{split}
\frac{1}{2\pi i} \oint_{\partial\mathbb{R}_y^{-}(k)}{\frac{dz}{w(z)\lambda_k-z}}	&=\lim_{z\rightarrow \psi(\lambda_k)}{\frac{\psi(\lambda_k)-z}{\lambda_k-f_F(z)}.\frac{f_F(z)}{\psi\left (f_F(z)\right )}}\\
&=\lim_{z\rightarrow \psi(\lambda_k)}{\frac{\psi(\lambda_k)-\psi\left( f_F(z)\right)}{\lambda_k-f_F(z)}.\frac{f_F(z)}{\psi\left (f_F(z)\right )}}\\
&=\frac{\lambda_k \psi '(\lambda_k)}{\psi(\lambda_k)}.
\end{split}
\end{equation*}
This implies,
\begin{equation*}
\frac{1}{2\pi i} \oint_{\partial\mathbb{R}_y^{-}(k)}{\frac{dz}{w(z)\lambda_j-z}}	=\begin{cases}
\frac{\lambda_k\psi'(\lambda_k)}{\psi(\lambda_k)} & j=k\\
0 & j\neq k
\end{cases}
\end{equation*}
and the proof is complete.
\end{proof}

\begin{proof}[\textbf{Proof of Lemma \ref{lemma:dct}}]
First, we show that $\hat{a}_1,\hat{a}_2,a_1,a_2$ can be chosen satisfying $\hat{a}_1\rightarrow a_1$ and $\hat{a}_2\rightarrow a_2$.
This is possible due to the fact that $d_k\xrightarrow{a.s.}\psi(\lambda_k)$ and $\psi(\lambda_k)\subset \psi\left (S_\psi,\infty\right )=(M_f,\infty)$ where $M_f>\sup{\Gamma_F}$. Therefore, we can choose a neighborhood $[a_1,a_2]$ around $\psi(\lambda_k)$ so that $[a_1,a_2]\subset (M_f,\infty)$. Moreover, as $M_f$ is bounded away from the support of the sample LSD $F$ and $d_k\xrightarrow{a.s.}\psi(\lambda_k)$, we can select a neighborhood $[\hat{a}_1,\hat{a}_2]$ around $d_k$ which does not contain any other eigenvalue for which $\hat{a}_1\rightarrow a_1, \hat{a}_2\rightarrow a_2$. Then,
\begin{multline}\label{eqn:dct}
\left | \frac{1}{2\pi i}\oint_{\partial\hat{\mathbb{R}}_y^{-}(k)}{\hat{m}_p(z)dz}- \frac{1}{2\pi i}\oint_{\partial\mathbb{R}_y^{-}(k)}{m_p(z)dz} \right |\\
\leq \frac{1}{2\pi}\left \{\sup_{z\in\partial\hat{\mathbb{R}}_y^{-}(k)\cup \partial\mathbb{R}_y^{+}(k)}{ \left |\hat{m}_p(z)\right |}\right \}\left(|\hat{a}_1-a_1)| +|\hat{a}_2-a_2)| \right )\\
+ \frac{1}{2\pi} \oint_{\partial\mathbb{R}_y^{-}(k)}{\left |\hat{m}_p(z)dz-m_p(z)\right ||dz|}.
\end{multline}
From Cauchy-Schwartz inequality, we can obtain the following upper bound for $\hat{m}_p$;
$$\left |\hat{m}_p(z)\right |\leq \frac{\|s_1 \|\|s_2 \|}{d\left( z,\Gamma_{F_p}\right )},$$ where $d\left( z,\Gamma_{F_p}\right )=\inf_{y\in\Gamma_{F_p}}{|z-y|}$. Since $F_p\rightarrow F$ point-wise and $[a_1,a_2]$ is bounded away from $\Gamma_F$, $d\left( z,\Gamma_{F_p}\right )$ is bounded away from zero with probability one for large enough $p$ and $n$. Therefore $\left |\hat{m}_p(z)\right |$ is finite for $z\in\mathbb{R}_y(k)$ with probability one for large enough $p$ and $n$. Moreover, since $[\hat{a}_1,\hat{a}_2]\rightarrow [a_1,a_2]$, the interval $[\hat{a}_1,\hat{a}_2]$ will eventually be bounded away from $\Gamma_F$. Thus, eventually the upper bound for $\left |\hat{m}_p(z)\right |$ will also be finite for $z\in\hat{\mathbb{R}}_y(k)$. Therefore, the first term on the right hand side of \eqref{eqn:dct} will go to zero as $\hat{a}_1 \rightarrow a_1, \hat{a}_2\rightarrow a_2$.

Now, as $\hat{m}_p(z)$ and $m_p(z)$ are holomorphic functions on the compact set $\partial\mathbb{R}_y^{-}(k)$,
$$ \sup_{z\in\partial\mathbb{R}_y^{-}(k)}{\left |\hat{m}_p(z)-m_p(z)\right |}<\infty.$$
Also from \eqref{eqn:girko}, $\left |\hat{m}_p(z)-m_p(z) \right |\xrightarrow{a.s.}0$ point-wise for all $z\in \mathbb{C\setminus R}$. Therefore, by dominated convergence theorem the second term on the right hand side of \eqref{eqn:dct} also converges to zero almost surely.
\end{proof}

We can show the asymptotic equivalence of the limits derived in Theorem \ref{thm:angle} and Result \ref{thm:angle1} as a direct application of the following lemma.
\begin{lemma}\label{lemma:bridge}
Suppose assumptions (A)--(C) hold. If $\lambda_k$ is a distant spike with multiplicity one, and $d_k$ is the corresponding sample eigenvalue, then
$$f_F(d_k)\xrightarrow{p}\lambda_k; \quad \frac{d_kg_F(d_k)}{f_F(d_k)}\xrightarrow{p}\psi'(\lambda_k).$$
\end{lemma}
\begin{proof}
	We have already established in the proof of Theorem \ref{thm:angle} that for all $z\in \mathbb{C^+}$ (similarly for $z\in \mathbb{C^-}$), $f_F(z)=f$ is the unique solution to \eqref{eqn:mp2angle} on $\mathbb{C^+}$ (respectively, $\mathbb{C^-}$). When $z$ is restricted to $\mathbb{C\setminus R}$, using \eqref{eqn:definef} and the fact that $b_F(z)=\int{(\tau-z)^{-1}dF(\tau)}$ we can write,
	$$f_F(z)=\frac{z}{1+\gamma\int{\frac{\tau dF(\tau)}{z-\tau}}}. $$
	Now suppose $z=x\in \mathbb{R}\setminus \{0\}$. Then both equations \eqref{eqn:mpangle} and \eqref{eqn:mp2angle} will have multiple roots (both real and complex valued depending on $x$ and $H$). If we look at \eqref{eqn:mp2angle} closely, we can see for real valued $x$ it can be represented as $\psi(f(x))=x$, where the $\psi$ function is as defined in \eqref{eqn:psif}. As we have seen in the proof of Theorem \ref{thm:spike}, $\psi$ is strictly increasing in the interval $(S_\psi,\infty)$ where $S_\psi>\sup{\Gamma_H}$ and $\psi ' (S_\psi)=0$. Therefore, any real-valued solution $f$ of $\psi(f(x))=x$ in $(S_\psi,\infty)$ has to be the inverse of $\psi$, which is unique due to the strict monotonicity of $\psi$ on $(S_\psi,\infty)$.
	Now suppose $\Gamma_F$ is the support of the sample LSD $F$. We will show that there exists $M_f>\sup{\Gamma_F}$ such that for any $x>M_f$, the function $f_F$ is real-valued and it is a solution to \eqref{eqn:mp2angle} in the interval $(S_\psi,\infty)$. Thus it is also the unique such solution and the inverse of the $\psi$ function in $(S_\psi,\infty)$.
	
	Let $x\in \mathbb{R},x>\sup{\Gamma_F}$ and $z=x+iy\in\mathbb{C^+}$. Now, as $z\in \mathbb{C^+}$, $f_F(z)$ is the unique solution to \eqref{eqn:mp2angle} in $\mathbb{C^+}$. Therefore, if we express $f_F(z)$ as $u(z)+iv(z)$, then $v(z)>0$. Also, the imaginary part of \eqref{eqn:mp2angle} can be written as
	$$v(z)\left [1-\gamma\int{\frac{\lambda^2}{\{u(z)-\lambda\}^2+v(z)^2}} \right ]=y.$$
	Both $v(z)$ and $y$ being positive implies that
	\begin{equation}\label{eqn:1stdegree}
	1-\gamma\int{\frac{\lambda^2}{\{u(z)-\lambda\}^2+v(z)^2}}>0.
	\end{equation}
	Due to the continuity of $f_F$ on the set $\{z\in \mathbb{C^+}: z=x+iy, x>\sup{\Gamma_F}\}$, $$f_F(x)=\lim_{y\rightarrow 0^+}{\frac{x+iy}{1+\gamma\int{\frac{\tau dF(\tau)}{x+iy-\tau}}}}=\frac{x}{1+\gamma\int{\frac{\tau dF(\tau)}{x-\tau}}},$$
	which is real-valued. Thus $u(z)\rightarrow f_F(x)$ and $v(z)\rightarrow 0$ as $y\rightarrow 0^+$. Therefore as $y\rightarrow 0^+$, the inequality \eqref{eqn:1stdegree} becomes
	$$1-\gamma\int{\frac{\lambda^2}{\{f_F(x)-\lambda\}^2}}>0,$$ which implies $\psi ' \left (f_F(x) \right )>0$.
	
	We can see that $f_F(x)$ attains zero at $\sup{\Gamma_F}$ and it is strictly and unboundedly increasing for $x>\sup{\Gamma_F}$. This ensures the existence of a threshold $M_F>\sup{\Gamma_F}$ such that the function $f_F$ maps the interval $(M_F,\infty)$ to $(S_\psi,\infty)$. Therefore, $f_F$ and $\psi$ are both strictly increasing, continuous and bijective mappings between the intervals $(M_F,\infty)$ and $(S_\psi,\infty)$.
	Since $f_F(z)$ is the unique solution to \eqref{eqn:mp2angle} in $\mathbb{C^+}$ when $z\in \mathbb{C^+}$, $f_F$ is also a solution to \eqref{eqn:mp2angle} in $(S_\psi,\infty)$ when $x>M_F$ due to the continuity of the left hand side of \eqref{eqn:mp2angle} on the set $\{f\in \mathbb{C^+}: f=u+iv, u>S_\psi\}$, which further implies that $f_F$ is the inverse function of $\psi$ on $(S_\psi,\infty)$.

	The first part of this lemma is proved as a corollary to Result \ref{thm:bai} as $\psi^{-1}=f_F$ on the domain of distant spikes, i.e. $(S_\psi,\infty)$. For the second part we first need to derive the expression of $f_F'$, and then derive the expression of $\psi '$ in terms of $f_F$ and $F$.
	\begin{equation*}
		f_F'(x)=\frac{f(x)}{x}\left [1+\gamma v_F(x) \right ]; \quad v_F(x)=\int{\frac{\tau dF(\tau)}{\left (x-\tau\right )^2}}.
	\end{equation*}
	For a distant spike $\lambda_k$, using the expression of $f_F'$ we get,
	\begin{equation*}
	\frac{\lambda_k\psi'(\lambda_k)}{\psi(\lambda_k)}=\frac{\lambda_k}{\psi(\lambda_k)f_F'\left (\psi(\lambda_k) \right )}=\frac{1}{1+\gamma f_F\left (\psi(\lambda_k) \right ) \int{\frac{\tau dF(\tau)}{\left [\psi(\lambda_k)-\tau\right ]^2}} }=g_F\left (\psi(\lambda_k)\right ).
	\end{equation*}
	As $\psi(\lambda_k)>M_f$, $g_F$ is continuous at $\psi(\lambda)$. Since $d_k\xrightarrow{p}\psi(\lambda_k)$, $$g_F(d_k)\xrightarrow{p}g_F\left (\psi(\lambda_k)\right )=\frac{\lambda_k\psi'(\lambda_k)}{\psi(\lambda_k)}; \quad \frac{d_kg_F(d_k)}{f_F(d_k)}\xrightarrow{p}\psi'(\lambda_k).$$
\end{proof}

\begin{proof}[\textbf{Proof of Theorem \ref{thm:score}}]
\begin{equation*}
\begin{split}
\left\langle P_k,p_k \right\rangle ^2 &=\frac{1}{n^2\lambda_kd_k}\left\langle XE_k,Xe_k \right\rangle ^2
=\frac{1}{\lambda_kd_k}\left (E_k^T\frac{X^TX}{n}e_k \right )^2\\
&=\frac{1}{\lambda_kd_k}\left [E_k^T\left (\sum_{i=1}^{p}{d_ie_ie_i^T}\right )e_k \right ]^2
=\frac{d_k}{\lambda_k}\left\langle e_k,E_k \right\rangle ^2.
\end{split}
\end{equation*}
Using the limits derived in Theorem \ref{thm:angle} and Result \ref{thm:bai},
	\begin{equation*}
	\left|\frac{d_k}{\lambda_k}\left\langle e_k,E_k \right\rangle ^2 -\psi'(\lambda_k)\right|\xrightarrow{p} 0.
	\end{equation*}
Using Lemma \ref{lemma:bridge},
	\begin{equation*}
	\left|\frac{d_k}{\lambda_k}\left\langle e_k,E_k \right\rangle ^2-\frac{d_kg_F(d_k)}{f_F(d_k)} \right|\xrightarrow{p} 0.
	\end{equation*}	
\end{proof}

\begin{proof}[\textbf{Proof of Theorem \ref{thm:shrinkage}}]
We show that the denominator $E\left (p_{kj}^2 \right )$ converges to $\psi(\lambda_k)$ and the numerator $E\left (q_k^2 \right )$ converges to $\lambda_k^2/\psi(\lambda_k)$. The proof will be complete using the fact that $d_k\xrightarrow{p}\psi(\lambda_k)$.

The denominator,
\begin{equation*}
\begin{split}
E\left (p_{kj}^2 \right )&=\frac{1}{n}E\left (\sum_{i=1}^{n}{p_{ki}^2} \right )
=\frac{1}{n}E\left (\sum_{i=1}^{n}{(x_i^Te_k)^2} \right )\\
&=E\left (e_k^T\frac{X^TX}{n}e_k\right )=E\left [e_k^T\left (\sum_{i=1}^{p}{d_ie_ie_i^T} \right )e_k\right ]=E(d_k)\rightarrow\psi(\lambda_k).
\end{split}
\end{equation*}

The numerator,
\begin{equation*}
\begin{split}
E\left (q_k^2 \right )&=E\left [(x_{new}^Te_k)^2 \right ]
=E\left [E(x_{new}^Te_k)^2|e_k \right ]\\
&=E\left [Var(x_{new}^Te_k)|e_k \right ]
=E\left [e_k^T\Sigma_pe_k\right ].
\end{split}
\end{equation*}

Now, using the notations in the proof of Theorem \ref{thm:angle} and Lemma \ref{lemma:bridge}, we have $b_F(z)=\int{(\tau-z)^{-1}dF(\tau)}$ as the Stieltjes transform of the sample LSD and the function $f_F$ defined as $f_F(z)=z\left [1-\gamma-\gamma z b_F(z) \right ]^{-1}$. Therefore,
\begin{equation*}
b_F(z)=\frac{(1-\gamma)f_F(z)-z}{\gamma z f_F(z)}.
\end{equation*}
The functions $b_F$ and $f_F$ can be extended to the real axis by defining the extensions as shown in the proof of Lemma \ref{lemma:bridge}. Thus, for the sample eigenvalue $d_k$ corresponding to the distant spike $\lambda_k$ we have
\begin{equation*}
b_F(d_k)=\frac{(1-\gamma)f_F(d_k)-d_k}{\gamma d_k f_F(d_k)}.
\end{equation*}
According to Theorem 4 in \cite{ledoit}, the limit of $e_k^T\Sigma_pe_k$ is given by $d_k\left [1-\gamma-\gamma d_k b_F(d_k) \right ]^{-2}$. Replacing the expression of $b_F(d_k)$ in this limit, we get
\begin{equation*}
\left |e_k^T\Sigma_p e_k - \frac{f_F^2(d_k)}{d_k}\right | \xrightarrow{p} 0.
\end{equation*}
Using Result \ref{thm:bai} and Lemma \ref{lemma:bridge} we have $f_F^2(d_k)/d_k\xrightarrow{p}\lambda_k^2/\psi(\lambda_k)$. Therefore, the limit of the numerator is given by,
$$E(q_k^2)=E\left [e_k^T\Sigma_pe_k\right ]\rightarrow\frac{\lambda_k^2}{\psi(\lambda_k)}.$$
\end{proof}

\begin{supplement}
\sname{Supplement A}\label{suppA}
\stitle{Supplementary tables and figures for ``Asymptotic properties of Principal Component Analysis and shrinkage-bias adjustment under the Generalized Spiked Population model"}
\slink[url]{SuppA.pdf}
\sdescription{Some additional tables and figures relevant to this paper are provided in this supplementary material.}
\end{supplement}

\nocite{*}
\bibliographystyle{imsart-nameyear}
\bibliography{ref1}

\newpage
\section{Tables and Figures}\label{ref:app_table}
\begin{table}[!htbp]
	\centering
	\caption{Simulation results for GSP-based and SP-based methods for estimating the population eigenvalues,  cosine of the angles between sample and population eigenvectors, correlations between  sample and population PC scores, and the asymptotic shrinkage factors. Each cell has empirical bias (\%)  with  coefficients of variations (\%) in parenthesis.}
	\label{tab:sim}
	\resizebox{\textwidth}{!}{%
		\begin{tabular}{ll|l|ll|ll|ll|ll}
				\hline
			\multicolumn{2}{l|}{\textbf{Settings}}                                                                                   & \textbf{Method} & \textbf{Eigenvalue}                                          & \textbf{}                                              & \multicolumn{2}{l|}{\textbf{Angle}}                                                                              & \multicolumn{2}{l|}{\textbf{Correlation}}                                                                        & \multicolumn{2}{l}{\textbf{Shrinkage}}                                                                         \\ \hline
			\textbf{No.} & \textbf{}                                                                                                 & \textbf{}       & \textbf{1}                                              & \textbf{2}                                             & \textbf{1}                                              & \textbf{2}                                             & \textbf{1}                                              & \textbf{2}                                             & \textbf{1}                                             & \textbf{2}                                             \\ \hline
			1            & \multirow{3}{*}{\begin{tabular}[c]{@{}l@{}}$n=500$\\ $p=5000$\\ $\sigma^2=4$\\ $\rho=0.8$\end{tabular}}   & SP              & \begin{tabular}[c]{@{}l@{}}5.27\\ (2.37)\end{tabular}   & \begin{tabular}[c]{@{}l@{}}18.27\\ (3.11)\end{tabular} & \begin{tabular}[c]{@{}l@{}}6.52\\ (0.32)\end{tabular}   & \begin{tabular}[c]{@{}l@{}}34.07\\ (0.60)\end{tabular} & \begin{tabular}[c]{@{}l@{}}3.83\\ (0.03)\end{tabular}   & \begin{tabular}[c]{@{}l@{}}23.33\\ (0.08)\end{tabular} & \begin{tabular}[c]{@{}l@{}}5.32\\ (0.60)\end{tabular}  & \begin{tabular}[c]{@{}l@{}}17.88\\ (1.06)\end{tabular} \\
			&                                                                                                           & $\lambda$-GSP   & \begin{tabular}[c]{@{}l@{}}0.43\\ (2.67)\end{tabular}   & \begin{tabular}[c]{@{}l@{}}0.95\\ (5.27)\end{tabular}  & \begin{tabular}[c]{@{}l@{}}0.53\\ (0.77)\end{tabular}   & \begin{tabular}[c]{@{}l@{}}3.28\\ (6.26)\end{tabular}  & \begin{tabular}[c]{@{}l@{}}0.33\\ (0.31)\end{tabular}   & \begin{tabular}[c]{@{}l@{}}2.79\\ (4.69)\end{tabular}  & \begin{tabular}[c]{@{}l@{}}0.47\\ (0.92)\end{tabular}  & \begin{tabular}[c]{@{}l@{}}0.58\\ (3.31)\end{tabular}  \\
			&                                                                                                           & $d$-GSP         & \begin{tabular}[c]{@{}l@{}}0.47\\ (2.67)\end{tabular}   & \begin{tabular}[c]{@{}l@{}}0.69\\ (5.45)\end{tabular}  & \begin{tabular}[c]{@{}l@{}}0.47\\ (0.77)\end{tabular}   & \begin{tabular}[c]{@{}l@{}}2.48\\ (6.70)\end{tabular}  & \begin{tabular}[c]{@{}l@{}}0.24\\ (0.31)\end{tabular}   & \begin{tabular}[c]{@{}l@{}}2.11\\ (5.07)\end{tabular}  & \begin{tabular}[c]{@{}l@{}}0.51\\ (0.92)\end{tabular}  & \begin{tabular}[c]{@{}l@{}}0.31\\ (3.51)\end{tabular}  \\ \hline
			2            & \multirow{3}{*}{\begin{tabular}[c]{@{}l@{}}$n=500$\\ $p=5000$\\ $\sigma^2=1$\\ $\rho=0.7$\end{tabular}}   & SP              & \begin{tabular}[c]{@{}l@{}}0.10\\ (0.90)\end{tabular}   & \begin{tabular}[c]{@{}l@{}}0.46\\ (1.27)\end{tabular}  & \begin{tabular}[c]{@{}l@{}}0.16\\ (0.04)\end{tabular}   & \begin{tabular}[c]{@{}l@{}}0.44\\ (0.08)\end{tabular}  & \begin{tabular}[c]{@{}l@{}}0.08\\ (0.001)\end{tabular}  & \begin{tabular}[c]{@{}l@{}}0.24\\ (0.003)\end{tabular} & \begin{tabular}[c]{@{}l@{}}0.18\\ (0.08)\end{tabular}  & \begin{tabular}[c]{@{}l@{}}0.39\\ (0.16)\end{tabular}  \\
			&                                                                                                           & $\lambda$-GSP   & \begin{tabular}[c]{@{}l@{}}-0.04\\ (0.90)\end{tabular}  & \begin{tabular}[c]{@{}l@{}}0.04\\ (1.28)\end{tabular}  & \begin{tabular}[c]{@{}l@{}}0.01\\ (0.04)\end{tabular}   & \begin{tabular}[c]{@{}l@{}}0.004\\ (0.10)\end{tabular} & \begin{tabular}[c]{@{}l@{}}0.01\\ (0.03)\end{tabular}   & \begin{tabular}[c]{@{}l@{}}0.01\\ (0.01)\end{tabular}  & \begin{tabular}[c]{@{}l@{}}0.03\\ (0.08)\end{tabular}  & \begin{tabular}[c]{@{}l@{}}-0.03\\ (0.18)\end{tabular} \\
			&                                                                                                           & $d$-GSP         & \begin{tabular}[c]{@{}l@{}}-0.004\\ (0.90)\end{tabular} & \begin{tabular}[c]{@{}l@{}}0.10\\ (1.28)\end{tabular}  & \begin{tabular}[c]{@{}l@{}}0.03\\ (0.04)\end{tabular}   & \begin{tabular}[c]{@{}l@{}}0.03\\ (0.10)\end{tabular}  & \begin{tabular}[c]{@{}l@{}}0.004\\ (0.03)\end{tabular}  & \begin{tabular}[c]{@{}l@{}}0.01\\ (0.01)\end{tabular}  & \begin{tabular}[c]{@{}l@{}}0.07\\ (0.08)\end{tabular}  & \begin{tabular}[c]{@{}l@{}}0.03\\ (0.18)\end{tabular}  \\ \hline
			3            & \multirow{3}{*}{\begin{tabular}[c]{@{}l@{}}$n=500$\\ $p=5000$\\ $\sigma^2=7.5$\\ $\rho=0.8$\end{tabular}} & SP              & \begin{tabular}[c]{@{}l@{}}25.68\\ (2.54)\end{tabular}  & \multicolumn{1}{c|}{-}                                 & \begin{tabular}[c]{@{}l@{}}64.06\\ (0.52)\end{tabular}  & \multicolumn{1}{c|}{-}                                 & \begin{tabular}[c]{@{}l@{}}46.50\\ (0.07)\end{tabular}  & \multicolumn{1}{c|}{-}                                 & \begin{tabular}[c]{@{}l@{}}26.41\\ (0.90)\end{tabular} & \multicolumn{1}{c}{-}                                  \\
			&                                                                                                           & $\lambda$-GSP   & \begin{tabular}[c]{@{}l@{}}2.92\\ (5.7)\end{tabular}    & \multicolumn{1}{c|}{-}                                 & \begin{tabular}[c]{@{}l@{}}12.62\\ (11.90)\end{tabular} & \multicolumn{1}{c|}{-}                                 & \begin{tabular}[c]{@{}l@{}}10.95\\ (10.13)\end{tabular} & \multicolumn{1}{c|}{-}                                 & \begin{tabular}[c]{@{}l@{}}3.47\\ (4.20)\end{tabular}  & \multicolumn{1}{c}{-}                                  \\
			&                                                                                                           & $d$-GSP         & \begin{tabular}[c]{@{}l@{}}2.45\\ (5.74)\end{tabular}   & \multicolumn{1}{c|}{-}                                 & \begin{tabular}[c]{@{}l@{}}12.25\\ (10.52)\end{tabular} & \multicolumn{1}{c|}{-}                                 & \begin{tabular}[c]{@{}l@{}}10.87\\ (8.58)\end{tabular}  & \multicolumn{1}{c|}{-}                                 & \begin{tabular}[c]{@{}l@{}}3.00\\ (4.24)\end{tabular}  & \multicolumn{1}{c}{-}                                  \\ \hline
			4            & \multirow{3}{*}{\begin{tabular}[c]{@{}l@{}}$n=500$\\ $p=5000$\\ $\sigma^2=4$\\ $\rho=0$\end{tabular}}     & SP              & \begin{tabular}[c]{@{}l@{}}0.05\\ (1.58)\end{tabular}   & \begin{tabular}[c]{@{}l@{}}-0.26\\ (2.35)\end{tabular} & \begin{tabular}[c]{@{}l@{}}0.06\\ (0.23)\end{tabular}   & \begin{tabular}[c]{@{}l@{}}-0.06\\ (0.53)\end{tabular} & \begin{tabular}[c]{@{}l@{}}0.03\\ (0.02)\end{tabular}   & \begin{tabular}[c]{@{}l@{}}0.05\\ (0.08)\end{tabular}  & \begin{tabular}[c]{@{}l@{}}0.07\\ (0.43)\end{tabular}  & \begin{tabular}[c]{@{}l@{}}-0.22\\ (0.90)\end{tabular} \\
			&                                                                                                           & $\lambda$-GSP   & \begin{tabular}[c]{@{}l@{}}0.03\\ (1.58)\end{tabular}   & \begin{tabular}[c]{@{}l@{}}-0.35\\ (2.35)\end{tabular} & \begin{tabular}[c]{@{}l@{}}0.02\\ (0.24)\end{tabular}   & \begin{tabular}[c]{@{}l@{}}-0.18\\ (0.54)\end{tabular} & \begin{tabular}[c]{@{}l@{}}0.01\\ (0.02)\end{tabular}   & \begin{tabular}[c]{@{}l@{}}-0.02\\ (0.09)\end{tabular} & \begin{tabular}[c]{@{}l@{}}0.04\\ (0.43)\end{tabular}  & \begin{tabular}[c]{@{}l@{}}-0.31\\ (0.91)\end{tabular} \\
			&                                                                                                           & $d$-GSP         & \begin{tabular}[c]{@{}l@{}}0.16\\ (1.58)\end{tabular}   & \begin{tabular}[c]{@{}l@{}}-0.12\\ (2.35)\end{tabular} & \begin{tabular}[c]{@{}l@{}}0.10\\ (0.23)\end{tabular}   & \begin{tabular}[c]{@{}l@{}}-0.03\\ (0.53)\end{tabular} & \begin{tabular}[c]{@{}l@{}}0.01\\ (0.02)\end{tabular}   & \begin{tabular}[c]{@{}l@{}}0.02\\ (0.09)\end{tabular}  & \begin{tabular}[c]{@{}l@{}}0.18\\ (0.42)\end{tabular}  & \begin{tabular}[c]{@{}l@{}}-0.08\\ (0.90)\end{tabular} \\ \hline
			5            & \multirow{3}{*}{\begin{tabular}[c]{@{}l@{}}$n=500$\\ $p=50000$\\ $\sigma^2=4$\\ $\rho=0.8$\end{tabular}}  & SP              & \begin{tabular}[c]{@{}l@{}}8.35\\ (5.75)\end{tabular}   & \begin{tabular}[c]{@{}l@{}}8.78\\ (5.01)\end{tabular}  & \begin{tabular}[c]{@{}l@{}}13.41\\ (1.71)\end{tabular}  & \begin{tabular}[c]{@{}l@{}}19.20\\ (1.64)\end{tabular} & \begin{tabular}[c]{@{}l@{}}8.98\\ (0.09)\end{tabular}   & \begin{tabular}[c]{@{}l@{}}13.53\\ (0.12)\end{tabular} & \begin{tabular}[c]{@{}l@{}}8.51\\ (3.23)\end{tabular}  & \begin{tabular}[c]{@{}l@{}}9.46\\ (3.04)\end{tabular}  \\
			&                                                                                                           & $\lambda$-GSP   & \begin{tabular}[c]{@{}l@{}}0.46\\ (6.80)\end{tabular}   & \begin{tabular}[c]{@{}l@{}}-3.38\\ (6.59)\end{tabular} & \begin{tabular}[c]{@{}l@{}}4.25\\ (2.99)\end{tabular}   & \begin{tabular}[c]{@{}l@{}}3.77\\ (3.76)\end{tabular}  & \begin{tabular}[c]{@{}l@{}}4.05\\ (0.85)\end{tabular}   & \begin{tabular}[c]{@{}l@{}}4.87\\ (1.46)\end{tabular}  & \begin{tabular}[c]{@{}l@{}}0.58\\ (4.28)\end{tabular}  & \begin{tabular}[c]{@{}l@{}}-2.82\\ (4.64)\end{tabular} \\
			&                                                                                                           & $d$-GSP         & \begin{tabular}[c]{@{}l@{}}1.35\\ (6.70)\end{tabular}   & \begin{tabular}[c]{@{}l@{}}-2.16\\ (6.45)\end{tabular} & \begin{tabular}[c]{@{}l@{}}4.96\\ (2.90)\end{tabular}   & \begin{tabular}[c]{@{}l@{}}4.84\\ (3.62)\end{tabular}  & \begin{tabular}[c]{@{}l@{}}4.29\\ (0.81)\end{tabular}   & \begin{tabular}[c]{@{}l@{}}5.29\\ (1.39)\end{tabular}  & \begin{tabular}[c]{@{}l@{}}1.48\\ (4.19)\end{tabular}  & \begin{tabular}[c]{@{}l@{}}-1.59\\ (4.50)\end{tabular}
				\\	\hline
		\end{tabular}%
	}
\end{table}
\newpage

\begin{figure}[!htbp]
	\caption{Eigenvalue structures for SP and GSP models}
	\centering
		\begin{subfigure}{.48\textwidth}
			\centering
			\caption{Spiked population}
		\includegraphics[width=0.82\textwidth]{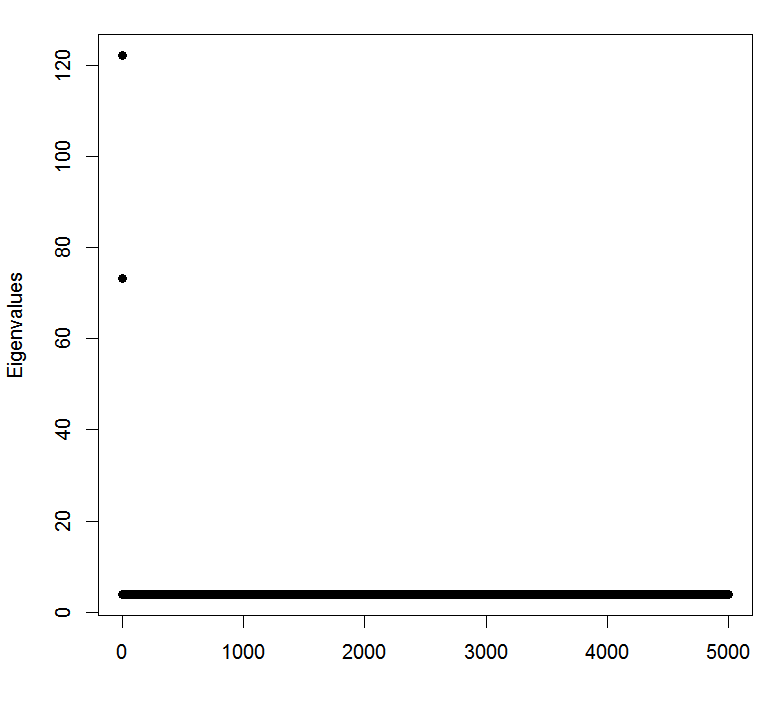}
		\label{fig:Spike}
		\end{subfigure}
		\begin{subfigure}{.48\textwidth}
			\centering
			\caption{Generalized spiked population}
		\includegraphics[width=0.82\textwidth]{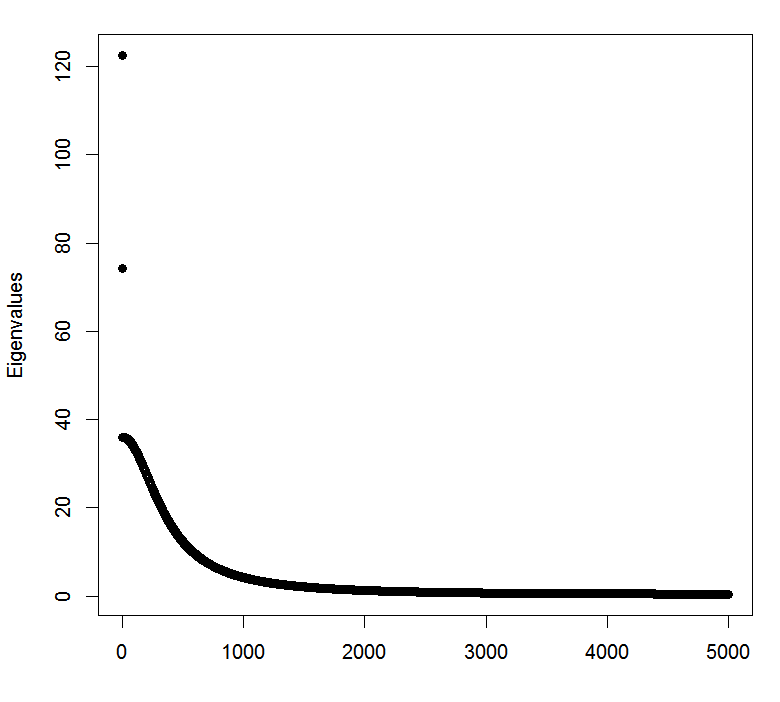}
		\label{fig:GSpike}
		\end{subfigure}
\end{figure}

\begin{figure}[!htbp]
	\caption{Eigenvalue structures in simulation studies}
	\label{fig:sim}
	\begin{subfigure}{.32\textwidth}
		\caption{Study 1}
		\includegraphics[width=\textwidth]{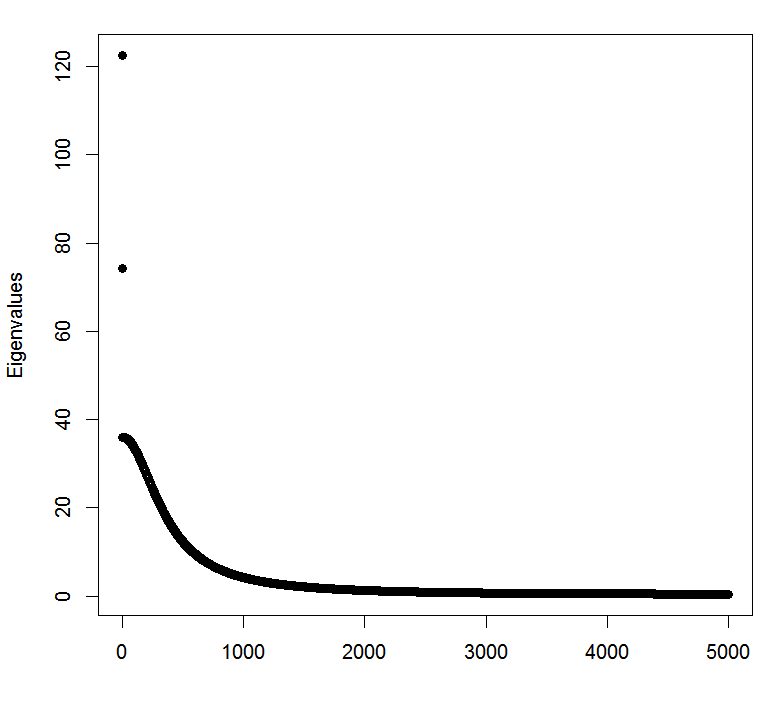}
	\end{subfigure}
	\begin{subfigure}{.32\textwidth}
		\caption{Study 2}
		\includegraphics[width=\textwidth]{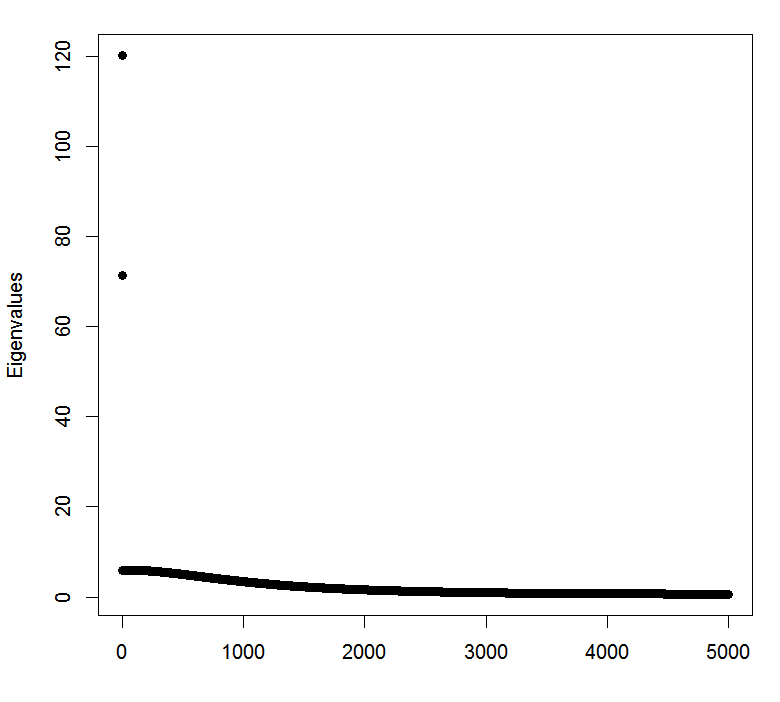}
	\end{subfigure}
	\begin{subfigure}{.32\textwidth}
		\caption{Study 3}
		\includegraphics[width=\textwidth]{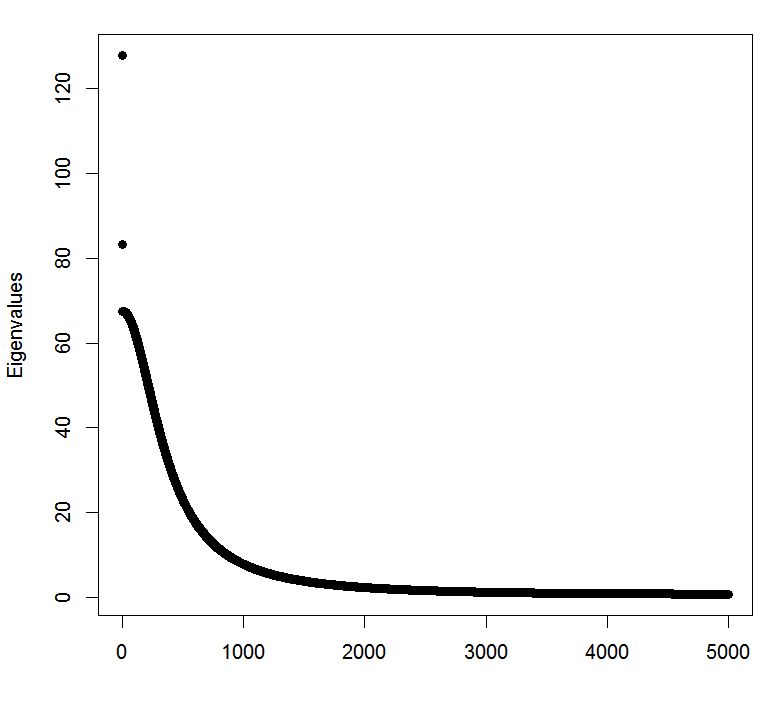}
	\end{subfigure}	
	\begin{subfigure}{.32\textwidth}
		\caption{Study 4}
		\includegraphics[width=\textwidth]{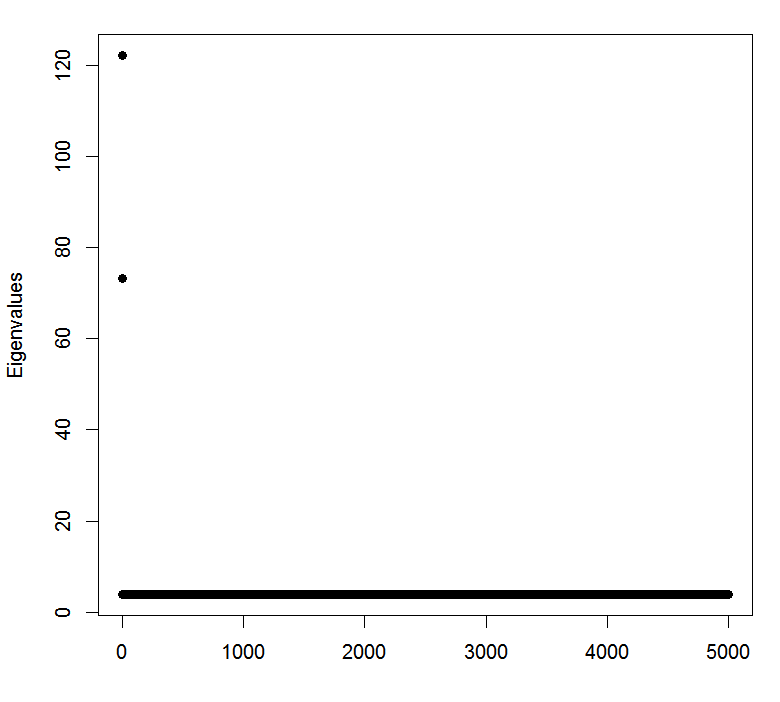}
	\end{subfigure}	
	\begin{subfigure}{.32\textwidth}
		\caption{Study 5}
		\includegraphics[width=\textwidth]{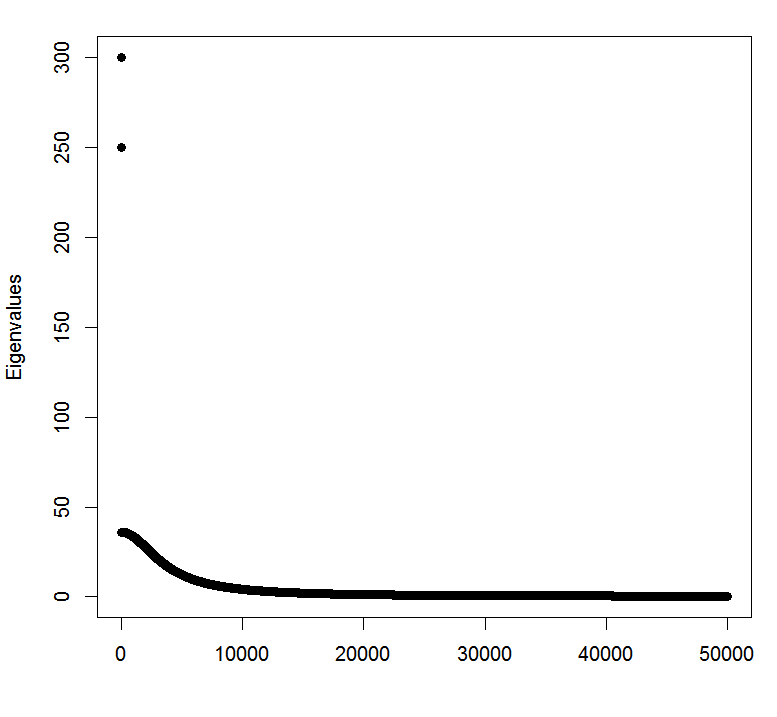}
	\end{subfigure}	
\end{figure}

\begin{figure}[!htbp]
	\caption{Estimated shrinkage factors for (a) PC1 and (b) PC2 across chromosomes 1-21 based on three different methods. (c) Comparison of the mean squared errors (MSE) of the unadjusted and adjusted PC scores based on the $d$-GSP and SP methods with the adjusted PC scores based on the $\lambda$-GSP method. The ratios of the MSEs are presented for chromosome 1-21 using the threshold $\epsilon=1$. The Y-Axis is presented in a logarithmic scale.}
	\begin{subfigure}{\textwidth}
		\centering
		\caption{}
		\label{fig:PC_hapmap1}
		\includegraphics[width=0.72\textwidth]{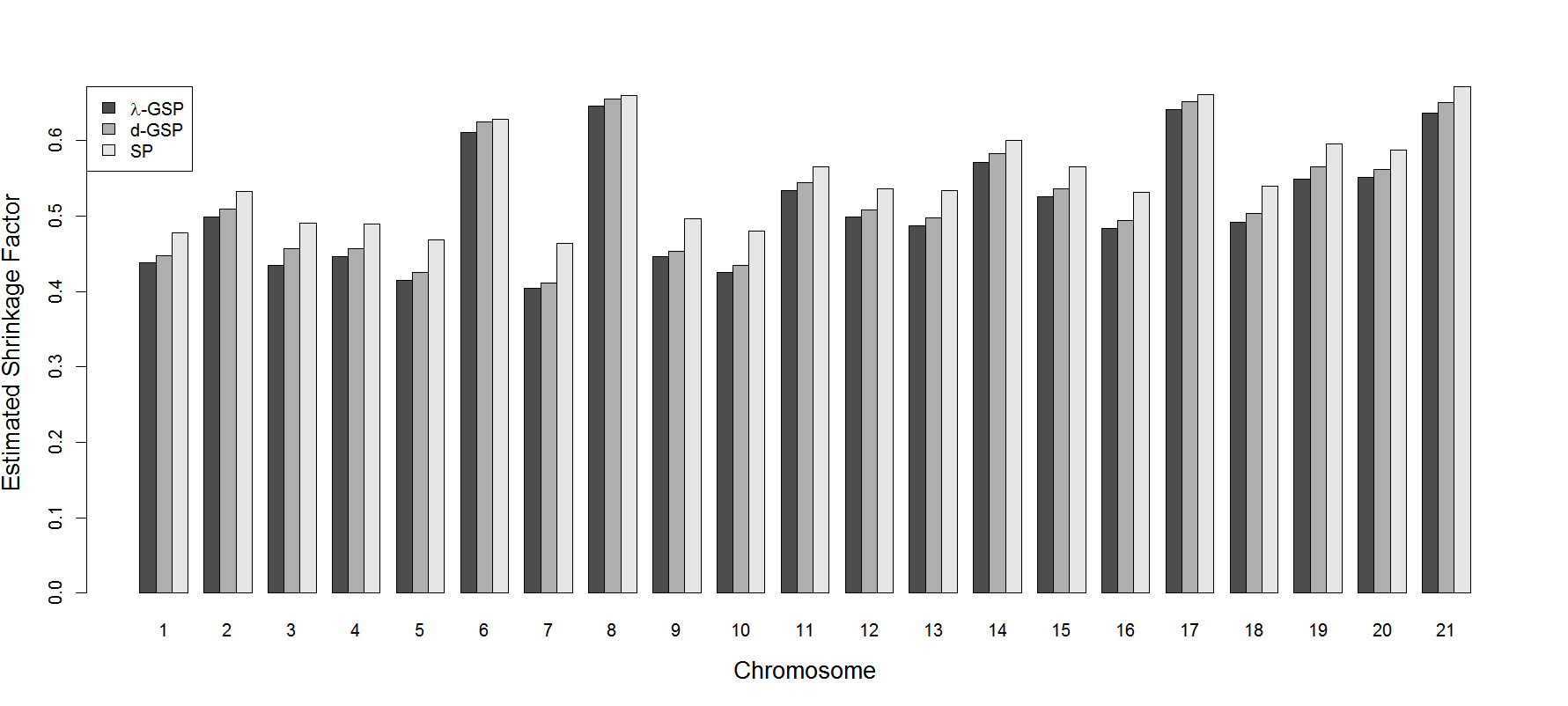}
	\end{subfigure}
	\begin{subfigure}{\textwidth}
		\centering
		\caption{}
		\label{fig:PC_hapmap2}
		\includegraphics[width=0.72\textwidth]{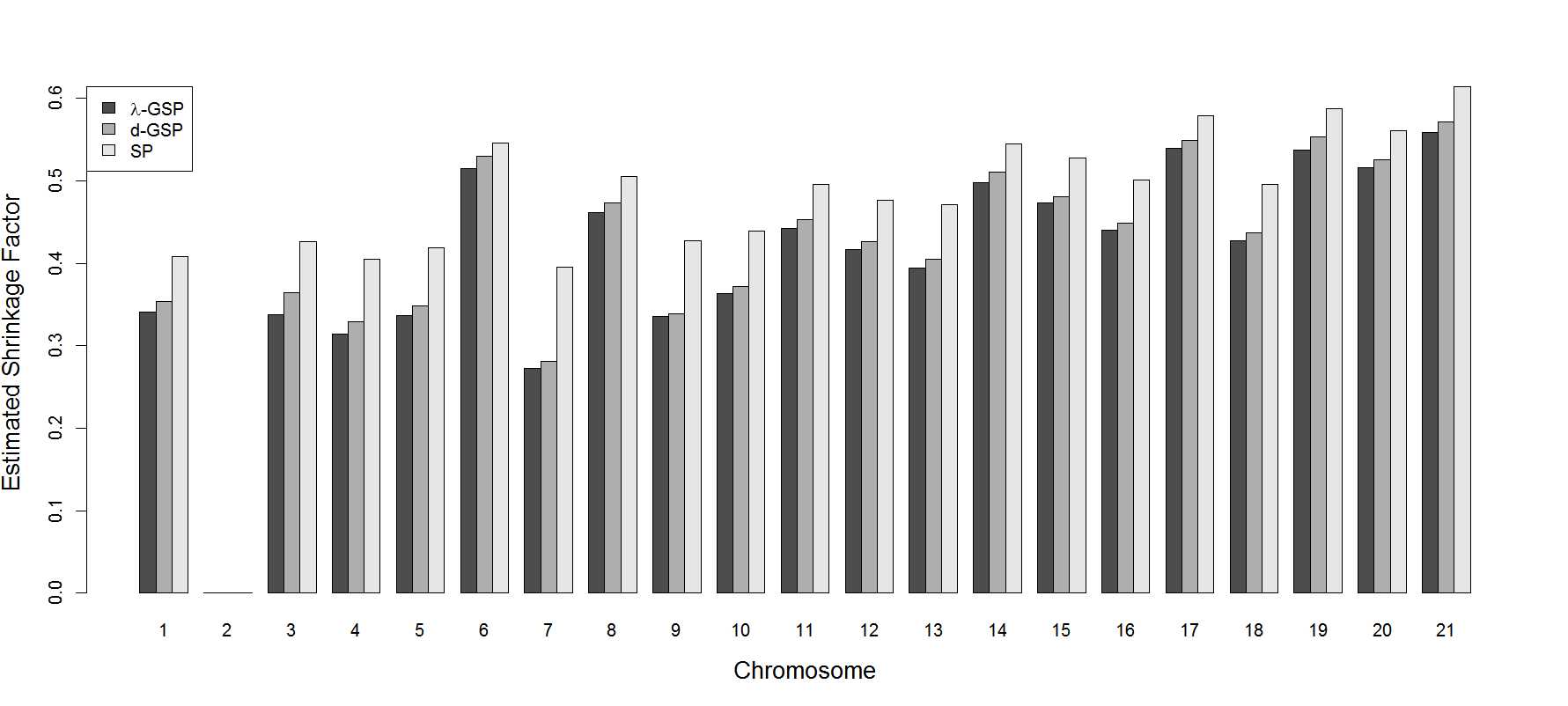}
	\end{subfigure}
	\begin{subfigure}{\textwidth}
		\centering
		\caption{}
	\label{fig:error_ratio}
	\includegraphics[width=0.72\textwidth]{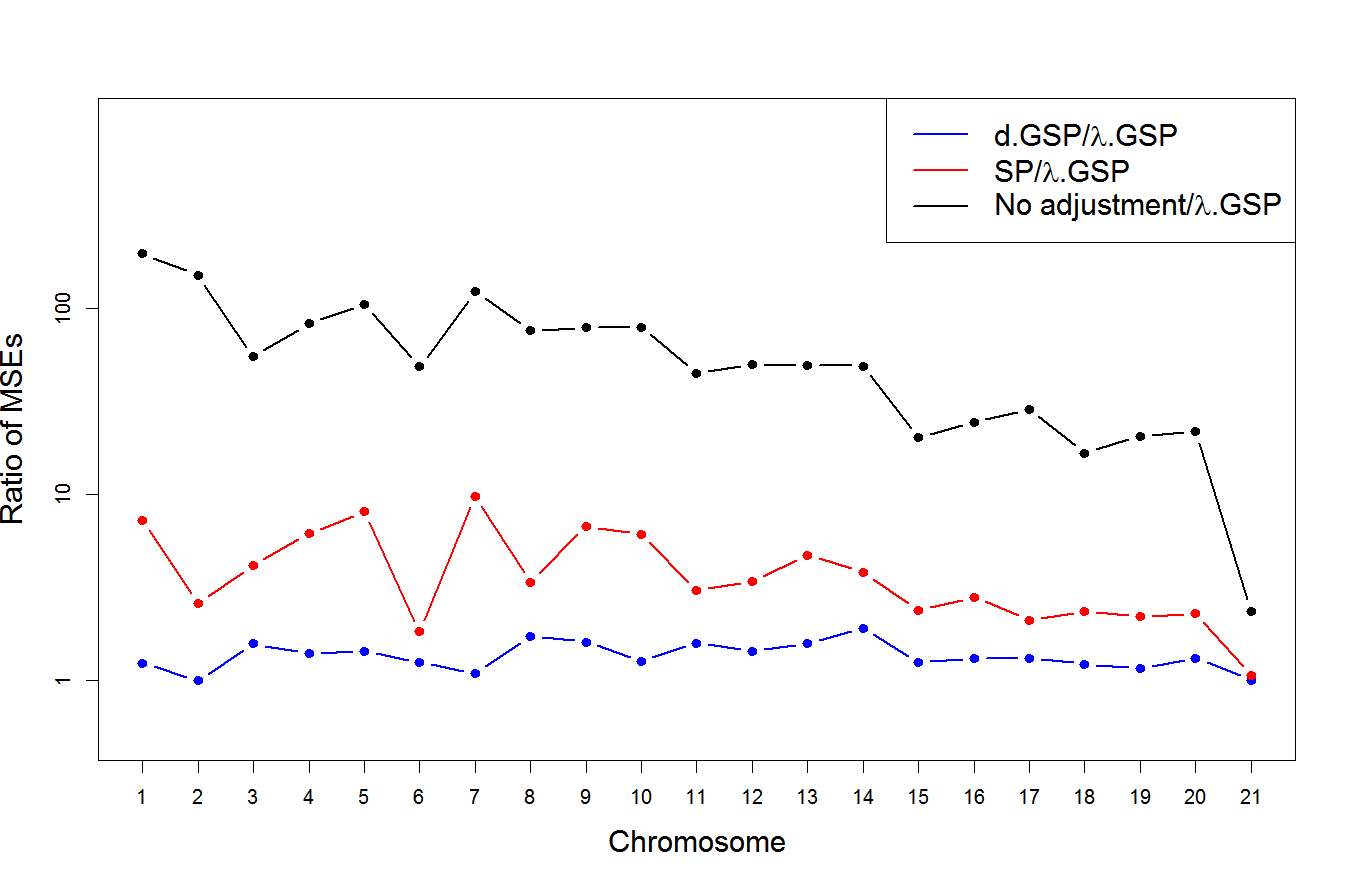}
	\end{subfigure}
\end{figure}

\begin{figure}[!htbp]
	\caption{PC1 vs PC2 plot of the Hapmap III CEU and TSI samples based on chromosome 7. The  predicted PC scores for the illustrative individual, and its bias-adjusted PC scores are also presented. 
	Since the $d$-GSP and the $\lambda$-GSP adjusted scores are nearly the same, the $\lambda$-GSP adjusted scores are not presented.}
	\begin{center}
		\includegraphics[width=\textwidth]{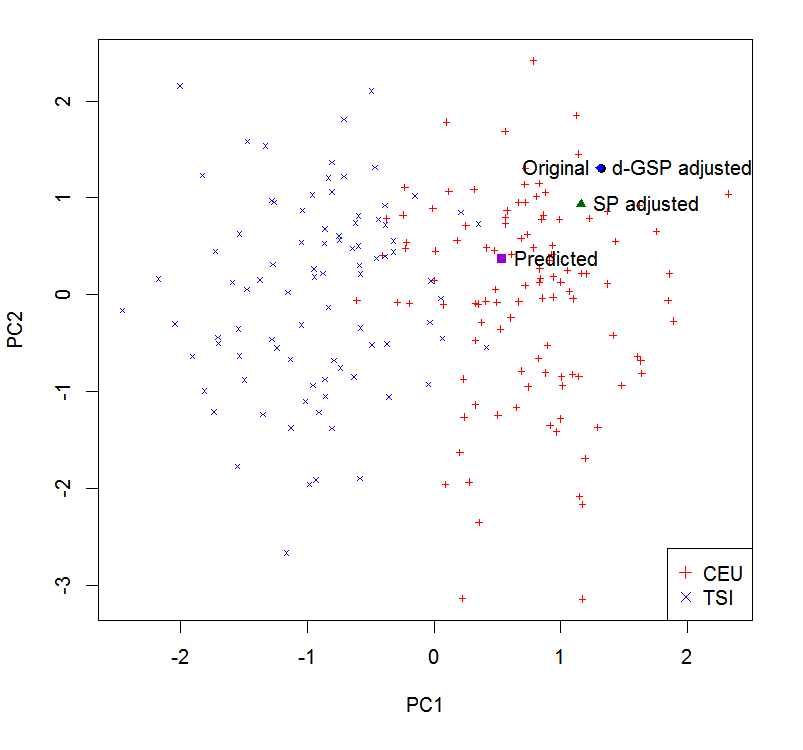}
	\end{center}
		\label{fig:bias_hapmap}
\end{figure}
\end{document}